\documentclass[a4paper,11pt,reqno]{amsart}

\usepackage{amsfonts}
\usepackage{amssymb}
\usepackage{amscd}
\usepackage{amsthm}
\usepackage{a4wide}
\usepackage{mathrsfs}
\usepackage[applemac]{inputenc}
\usepackage{amsmath}
\usepackage{amssymb}
\usepackage{amsthm}
\usepackage{textcomp}
\usepackage{graphicx}
\usepackage{enumerate}
\usepackage{mathrsfs}
\usepackage{frcursive}
\usepackage{tikz}
\usepackage[cyr]{aeguill}
\usepackage{xspace}
\usepackage{hyperref}
\usepackage{appendix}
\usepackage{esint}
\usepackage{cancel}
\usepackage{calc}

\newtheorem{defn}{Definition}[section]
\newtheorem{lemma}[defn]{Lemma}
\newtheorem{prop}[defn]{Proposition}
\newtheorem{theo}[defn]{Theorem}

\newtheorem{rk}[defn]{Remark}

\def\Ric{\mathop{\rm Ric}\nolimits}

\def\Rm{\mathop{\rm Rm}\nolimits}
\def\tr{\mathop{\rm tr}\nolimits}

\def\vol{\mathop{\rm vol}\nolimits}
\def\eucl{\mathop{\rm eucl}\nolimits}
\def\dim{\mathop{\rm dim}\nolimits}
\def\vol{\mathop{\rm Vol}\nolimits}

\def\div{\mathop{\rm div}\nolimits}

\def\AVR{\mathop{\rm AVR}\nolimits}

\def\Li{\mathop{\rm \mathscr{L}}\nolimits}

\def\Ric{\mathop{\rm Ric}\nolimits}

\def\Rm{\mathop{\rm Rm}\nolimits}
\def\tr{\mathop{\rm tr}\nolimits}

\def\vol{\mathop{\rm vol}\nolimits}
\def\eucl{\mathop{\rm eucl}\nolimits}
\def\dim{\mathop{\rm dim}\nolimits}
\def\vol{\mathop{\rm Vol}\nolimits}

\def\div{\mathop{\rm div}\nolimits}

\def\AVR{\mathop{\rm AVR}\nolimits}

\def\Li{\mathop{\rm \mathscr{L}}\nolimits}

\def\supp{\mathop{\rm supp}\nolimits}

\def\R{\mathop{\rm R}\nolimits}

\def\RP{\mathop{\rm \mathcal{R}}\nolimits}
\def\SO{\mathop{\rm SO}\nolimits}
\def\ind{\mathop{\rm ind}\nolimits}

\newsavebox\CBox
\newcommand\hcancel[2][0.5pt]{%
  \ifmmode\sbox\CBox{$#2$}\else\sbox\CBox{#2}\fi%
  \makebox[0pt][l]{\usebox\CBox}%  
  \rule[0.5\ht\CBox-#1/2]{\wd\CBox}{#1}}

\title{Stability of ALE Ricci-flat manifolds under Ricci flow}

\author[Alix Deruelle and Klaus Kr\"{o}ncke]{Alix Deruelle and Klaus Kr\"{o}ncke}
\address[Alix Deruelle]{Institut de math\'ematiques de Jussieu, 4, place Jussieu, Boite Courrier 247 - 75252 Paris}
\email{alix.deruelle@imj-prg.fr}
\address[Klaus Kr\"{o}ncke]{Fachbereich Mathematik, Universit\"{a}t Hamburg, Bundesstra{\ss}e 55, D-20146, Hamburg }
\email{klaus.kroencke@uni-hamburg.de}

\begin{document}
\begin{abstract}We prove that if an ALE Ricci-flat manifold $(M,g)$ is linearly stable and integrable, it is dynamically stable under Ricci flow, i.e. any Ricci flow starting close to $g$ exists for all time and converges modulo diffeomorphism to an ALE Ricci-flat metric close to $g$. By adapting Tian's approach in the closed case, we show that integrability holds for ALE Calabi-Yau manifolds which implies that they are dynamically stable.
\end{abstract}
\maketitle
\section{Introduction}
%A gravitational instanton is a complete noncompact four-dimensional hyperk\"{a}hler manifold $(M,g)$ whose curvature decays faster than quadratic.
Consider a complete Riemannian manifold $(M^n,g)$ without boundary endowed with a Ricci-flat metric $g$. As such, it is a fixed point of the Ricci flow and therefore, it is a natural problem to study the stability of such a metric with respect to the Ricci flow. Whether the manifold is compact or noncompact makes an essential difference in the analysis. In both cases, if $(M^n,g)$ is Ricci-flat, the linearized operator is the so called Lichnerowicz operator 
%$\Delta_g+2\Rm(g)\ast$
 acting on symmetric $2$-tensors. Nonetheless, the $L^2$ approach differs drastically in the noncompact case. Indeed, even in the simplest situation, that is the flat case, the spectrum of the Lichnerowicz operator is not discrete anymore and $0$ belongs to the essential spectrum. In this paper, we consider Ricci-flat metrics on noncompact manifolds that are asymptotically locally Euclidean (ALE for short), i.e. that are asymptotic to a flat cone over a space form $\mathbb{S}^{n-1}/\Gamma$ where $\Gamma$ is a finite group of $\SO(n)$ acting freely on $\mathbb{R}^n\setminus\{0\}$.

If $(M^n,g_0)$ is an ALE Ricci-flat metric, we assume furthermore that it is linearly stable, i.e. the Lichnerowicz operator is nonpositive in the $L^2$ sense and that the set of ALE Ricci-flat  metrics close to $g_0$ is integrable, i.e. has a manifold structure of finite dimension: see Section \ref{def-ALE}.

The strategy we adopt is given by Koch and Lamm \cite{Koc-Lam-Rou} that studied the stability of the Euclidean space along the Ricci flow in the $L^{\infty}$ sense. The quasi-linear evolution equation to consider here is 
\begin{eqnarray}
\partial_tg=-2\Ric_g+\Li_{V(g,g_0)}(g),\quad\text{$\|g(0)-g_0\|_{L^{\infty}(g_0)}$ small} ,\label{eq-ricci-flat}
\end{eqnarray}
where $(M^n,g_0)$ is a fixed background ALE Ricci-flat metric and $\Li_{V(g,g_0)}(g)$ is the so called DeTurck's term. Equation (\ref{eq-ricci-flat}) is called the Ricci-DeTurck flow: its advantage over the Ricci flow equation is to be a strictly parabolic equation instead of a degenerate one. Koch and Lamm managed to rewrite (\ref{eq-ricci-flat}) in a clever way to get optimal results regarding the regularity of the initial condition: see Section \ref{RF-Sec}.

 Our main theorem is then:
\begin{theo}\label{main-theo-bis}
	Let $(M^n,g_0)$ be an ALE Ricci-flat space. Assume it is linearly stable and integrable. Then for every $\epsilon>0$, there exists a $\delta>0$ such that the following holds: for any metric $g\in \mathcal{B}_{L^2\cap L^{\infty}}(g_0,\delta)$, there is a complete Ricci-DeTurck flow $(M^n,g(t))_{t\geq 0}$ starting from $g$ converging to an ALE Ricci-flat metric $g_{\infty}\in \mathcal{B}_{L^2\cap L^{\infty}}(g_0,\epsilon)$. 
	
	Moreover, the $L^{\infty}$ norm of $(g(t)-g_0)_{t\geq 0}$ is decaying sharply at infinity:
	\begin{eqnarray*}
\|g(t)-g_0\|_{L^{\infty}(M\setminus B_{g_0}(x_0,\sqrt{t}))}\leq C(n,g_0,\epsilon)\frac{\sup_{t\geq 0}\|g(t)-g_0\|_{L^2(M)}}{t^{\frac{n}{4}}},\quad t>0.
\end{eqnarray*} 
\end{theo}
   As far as we know, this theorem is the first stability result for non-flat and noncompact Ricci-flat manifolds under Ricci flow.\\
   
 Schn\" urer, Schulze and Simon \cite{Sch-Sch-Sim} have proved the stability of the Euclidean space for an $L^2\cap L^{\infty}$ perturbation as well. The decay obtained in Theorem \ref{main-theo-bis} sharpens their result. Indeed, the proof shows that if $(M^n,g_0)$ is isometric to $(\mathbb{R}^n,\eucl)$ then the $L^{\infty}$ decay holds on the whole manifold. This $L^{\infty}$ decay on Euclidean space was also recently shown in \cite{Appleton} by a different approach.

\begin{rk}
It is an open question whether the decay in time obtained in Theorem \ref{main-theo-bis}  holds on the whole manifold with an exponent $\alpha$ less than or equal to $n/4$.
\end{rk}

One of the main difficulties in proving Theorem \ref{main-theo-bis} is to establish a uniform-in-time $L^2$ bound on the difference of the metrics $g(t)-g_0(t)$ for non-negative time $t$ and a suitable family of Ricci-flat reference metrics $g_0(t)$ . To prove such a bound in case the background metric $g_0$ is flat, \cite{Sch-Sch-Sim} and \cite{Appleton} use a direct integration by parts. In our setting, this approach does not work for mainly two reasons: the kernel of the Lichnerowicz operator might be non-trivial and the curvature terms from the linearized operator cannot be treated as error terms. Therefore, the strategy is to work orthogonally to the kernel of the Lichnerowicz operator in order to apply a delicate notion of strict positivity for the Lichnerowicz operator that enables us to absorb the curvature terms. This leads us in turn to a very delicate choice of a reference metric $g_0(t)$ which makes the analysis trickier: see Section \ref{deco-met-sec}.  \\

 Now, from the physicist point of view, the question of stability of ALE Ricci-flat  metrics is of great importance when applied to hyperk\"{a}hler or Calabi-Yau ALE metrics: Hyperk\"{a}hler ALE metrics (also called gravitational instantons) are of great importance in theoretical physics: see for instance \cite{Ber-Cam-Fer-Fre}, \cite{Ans-Dam-Bil-Mar-Fre} and the references therein. In this context, the Ricci flow serves as a first-order approximation of the renormalization group flow and admits deep connections with it, see \cite{Car-Ren-Grp}.
 
  In the hyperkähler or Calabi-Yau case, the Lichnerowicz operator is always a nonnegative operator because of the special algebraic structure of the curvature tensor shared by these metrics. It turns out that they are also integrable: see Theorem \ref{CY-ALE-Int} based on the fundamental results of Tian \cite{Tian-Smoothness} in the closed case. In particular, it gives us plenty of examples to which one can apply Theorem \ref{main-theo-bis}. 
 
 Another source of motivation comes from the question of continuing the Ricci flow after it reached a finite time singularity on a $4$-dimensional closed Riemannian manifold: the works of Bamler and Zhang \cite{Bam-Zha-I} and Simon \cite{Sim-Ext} show that the singularities that can eventually show up for Ricci flows with bounded scalar curvature are orbifold singularities and thus modelled over ALE Ricci-flat metrics. A strong connection between the appereance of such singularities and the stability of their blow-up limits is expected. However, there is no classification available of such metrics in dimension $4$ at the moment, except  Kronheimer's classification for hyperk\"ahler metrics \cite{Kron-Class}. \\

Finally, we would like to discuss some related results especially regarding the stability of closed Ricci-flat metrics. There have been basically two approaches. On one hand, Sesum \cite{Ses-Lin-Dyn-Sta} has proved the stability of integrable Ricci-flat metrics on closed manifolds: in this case, the convergence rate is exponential since the spectrum of the Lichnerowicz operator is discrete. On the other hand, Haslhofer-M\"uller \cite{Has-Mul-Lam} and the second author \cite{Kro-Sta-Ric-Sol} have proved Lojasiewicz inequalities for Perelman's entropies which are monotone under the Ricci flow and whose critical points are exactly Ricci-flat metrics and Ricci solitons, respectively. The analysis in the proof of Theorem \ref{main-theo-bis} differs substantially from these two previous approaches as these tools and features are not available in our setting.\\

The paper is organized as follows. Section \ref{def-ALE} recalls the basic definitions of ALE spaces together with the notions of linear stability and integrability of a Ricci-flat metric. Section \ref{sec-moduli-space} gives a detailed description of the space of gauged ALE Ricci-flat metrics: see Theorem \ref{ell-reg-prop} and Theorem \ref{analyticset}. Section \ref{Ricci-Flat-ALE-Kahler} investigates the integrability of K\"ahler Ricci-flat metrics: this is the content of Theorem \ref{CY-ALE-Int}. Section \ref{RF-Sec} is devoted to the proof of the first part of Theorem \ref{main-theo-bis}. Section \ref{sec-equ-flo} discusses the structure of the Ricci-DeTurck flow. Section \ref{Short-time} establishes pointwise and integral short time estimates. The core of the proof of Theorem \ref{main-theo-bis} is contained in Section \ref{deco-met-sec}: a priori uniform in time $L^2$ estimates are proved with the help of a suitable notion of strict positivity for the Lichnerowicz operator developed for Schr\"odinger operators by Devyver \cite{Dev-Gau-Est}. The infinite time existence and the convergence aspects of Theorem \ref{main-theo-bis} are then proved in Section \ref{Exi-conv}. 
Finally, Section \ref{Nash-Moser-Sec} proves the last part of Theorem \ref{main-theo-bis}: the decay in time is verified with the help of a Nash-Moser iteration.
\subsection*{Acknowledgements} This work was carried out while the authors were supported by the National Science Foundation under Grant No.~DMS-1440140 while in residence at the Mathematical Sciences Research Institute (MSRI) in Berkeley, California, during the Spring 2016 semester. They wish to thank MSRI for their excellent working conditions and hospitality during this time.

\section{ALE spaces}\label{Sec-ALE}
\subsection{Analysis on ALE spaces}\label{def-ALE}
We start by recalling a couple of definitions. 

\begin{defn}
	A complete Riemannian manifold $(M^n,g_0)$ is said to be \textbf{asymptotically locally Euclidean} (ALE) with one end of order $\tau>0$ if there exists a compact set $K\subset M$, a radius $R$ and a diffeomorphism such that : $\phi:M\setminus K\rightarrow (\mathbb{R}^n\setminus B_R)/\Gamma$, where $\Gamma$ is a finite group of $\SO(n)$ acting freely on $\mathbb{R}^n\setminus\{0\}$, such that 
	\begin{eqnarray*}
		\arrowvert \nabla^{\eucl,k}(\phi_*g_0-g_{\eucl})\arrowvert_{\eucl}=\textit{O}(r^{-\tau-k}),\quad \forall k\geq 0.
	\end{eqnarray*}
	holds on $(\mathbb{R}^n\setminus B_R)/\Gamma$.
\end{defn}
%This definition can be generalized to multiple ends in an obvious manner.
% and all the results in this paper also hold for this more general case.

The linearized operator we will focus on is the so called Lichnerowicz operator whose definition is recalled below:
\begin{defn}
	Let $(M,g)$ be a Riemannian manifold. Then the operator $L_g:C^{\infty}(S^2T^*M)\rightarrow C^{\infty}(S^2T^*M)$, defined by
	\begin{eqnarray*}
		L_g(h)&:=&\Delta_gh+2\Rm(g)\ast h-\Ric(g)\circ h-h\circ \Ric(g),\\
		(\Rm(g)\ast h)_{ij}&:=&\Rm(g)_{iklj}h_{mn}g^{km}g^{ln},\\
		(h\circ k)_{ij}&:=&h_{ik}k_{lj}g^{kl},\quad k\in S^2T^*M,
	\end{eqnarray*}
	is called the \textbf{Lichnerowicz} Laplacian acting on the space of symmetric $2$-tensors $S^2T^*M$.
	
\end{defn}
\noindent In this paper, we consider the following notion of stability:
\begin{defn}
	Let $(M^n,g_0)$ be a complete ALE Ricci-flat manifold. $(M^n,g_0)$ is said to be \textbf{linearly stable} if the (essential) $L^2$ spectrum of the Lichnerowicz operator $L_{g_0}:=\Delta_{g_0}+2\Rm(g_0)\ast$ is in $(-\infty,0]$. 
\end{defn}
Equivalently, this amounts to say that $\sigma_{L^2}(-L_{g_0})\subset [0,+\infty)$.
By a theorem due to G. Carron \cite{Car-Coh}, $\ker_{L^2}L_{g_0}$ has finite dimension. Denote by $\Pi_c$ the $L^2$ projection on the kernel $\ker_{L^2}L_{g_0}$ and $\Pi_s$ the projection orthogonal to $\Pi_c$ so that $h=\Pi_ch+\Pi_sh$ for any $h\in L^2(S^2T^*M)$.

Let $(M,g_0)$ be an ALE Ricci-flat manifold and $\mathcal{U}_{g_0}$ the set of ALE Ricci-flat metrics with respect to the gauge $g_0$, that is : 
\begin{eqnarray}
\mathcal{U}_{g_0}&:=&\left\{\mbox{$g \mid$ $g$ ALE metric on $M$ s.t.}\quad \Ric(g)=0\mbox{ and }\Li_{V(g,g_0)}(g)=0\right\},\\
g_0(V(g,g_0),.)&:=&\div_{g}g_0-\frac{1}{2}\nabla^{g}\tr_{g}g_0,\label{def-vect-gauge}
\end{eqnarray}
endowed with the $L^2\cap L^{\infty}$ topology coming from $g_0$.

\begin{defn}\label{defn-integrable}
	$(M^n,g_0)$ is said to be \textbf{integrable} if $\mathcal{U}_{g_0}$ has a smooth structure in a neighborhood of $g_0$. In other words, $(M^n,g_0)$ is integrable if the map 
	\begin{eqnarray*}
		\Psi_{g_0}: g\in \mathcal{U}_{g_0}\rightarrow \Pi_c(g-g_0)\in \ker_{L^2}(L_{g_0}),
	\end{eqnarray*}
	is a local diffeomorphism at $g_0$. 
\end{defn}
If $(M,g_0)$ is ALE and Ricci-flat, it is a consequence of \cite[Theorem 1.1]{Ban-Kas-Nak} that it is already ALE of order $n-1$. Moreover, if $n=4$ or $(M,g_0)$ is K\"ahler, it is ALE of order $n$. This is due to the presence of Kato inequalities, \cite[Corollary 4.10]{Ban-Kas-Nak} for the curvature tensor.
We will show in Theorem \ref {ell-reg-prop} that by elliptic regularity, all $g\in\mathcal{U}_{g_0}$ are ALE of order $n-1$ with respect to the same coordinates as $g_0$. 

In order to do analysis of partial differential equations on ALE manifolds, one has to work with weighted function spaces which we will define in the following. Fix a point $x\in M$ and define a function $\rho:M\to\mathbb{R}$ by  $\rho(y)=\sqrt{1+d(x,y)^2}$. For $p\in[1,\infty)$ and $\delta\in \mathbb{R}$, we define the spaces  $L^p_{\delta}(M)$ as the closure of $C^{\infty}_{0}(M)$ with respect to the norm
\begin{align*}
\left\|u\right\|_{L^p_{\delta}}=\left(\int_M |\rho^{-\delta}u|^p\rho^{-n}d\mu\right)^{1/p},
\end{align*}
and the weighted Sobolev spaces $W^{k,p}_{\delta}(M)$ as the closure of $C^{\infty}_{0}(M)$ under
\begin{align*}
\left\|u\right\|_{W^{k,p}_{\delta}}=\sum_{l=0}^k\left\|\nabla^lu\right\|_{L^{p}_{\delta-l}}.
\end{align*}
The weighted H\"{o}lder spaces are defined as the set of maps $u\in C^{k,\alpha}_{loc}(M)$, $\alpha\in(0,1)$ such that the following quantity, 
\begin{align*}
\left\|u\right\|_{C^{k,\alpha}_{\delta}}=&\sum_{l=0}^k\sup_{x\in M} \rho^{-\delta+l}(x)|\nabla^lu(x)|\\&\quad+\sup_{\substack{x, y\in M\\ 0<d(x,y)<\mathrm{inj}(M)}}\min\left\{\rho^{-\delta+k+\alpha}(x),\rho^{-\delta+k+\alpha}(y)\right\}\frac{|\tau^y_x\nabla^ku(x)-\nabla^k u(y)|}{|x-y|^{\alpha}},
\end{align*}
is finite. Here $\tau_x^y$ denotes the parallel transport from $x$ to $y$ along the shortest geodesic joining $x$ and $y$. All these spaces are Banach spaces, the spaces $H^k_{\delta}(M):=W^{k,2}_{\delta}(M)$ are Hilbert spaces and their definition does not depend on the choice of the base point defining the weight function $\rho$. All these definitions extend to Riemannian vector bundles with a metric connection in an obvious manner.

In the literature, there are different notational conventions for weighted spaces. We follow the convention of \cite{Bar-Mass}.
The Laplacian $\Delta_g$ is a bounded map
$\Delta_g:W^{p,k}_{\delta}(M)\to W^{p,k-2}_{\delta-2}(M)$ and there exists a discrete set $D\subset\mathbb{R}$ such that this operator is Fredholm for $\delta\in \mathbb{R}\setminus D$. This is shown in \cite{Bar-Mass} in the asymptotically flat case and the proof in the ALE case is the same up to minor changes. We call the values $\delta\in D$ exceptional and the values $\delta\in \mathbb{R}\setminus D$ nonexceptional. If $\delta\in (2-n,0)$, the operator is even an isomorphism  \cite[p.\ 151]{Pac-T}.
The Fredholm properties also hold for elliptic operators of arbitrary order acting on vector bundles supposed that the coefficients behave suitable at infinity \cite[Theorem 6.1]{Loc-McO}. The isomorphism properties also hold with the same range of $\delta$ for connection Laplacians on arbitrary tensor bundles.
%
%\textcolor{blue}{These facts also hold in the ALF case. 
%	The Fredholm properties of elliptic operators follow from the scale-broken estimate in [Minerbe, Proposition 1], which can be easily generalized to elliptic operators on arbitrary vector bundles, and an application of Rellich's lemma, c.f. the arguments in [Bartnik, Theorem 1.10].}
We will use these facts frequently in this paper.

%
%
%and let us denote the corresponding spaces by $W^{k,p}_{\delta}(M)$.
%We need this weighted versions in order to ensure that $L_{g_0}$ is Fredholm. In fact, on $\mathbb{R}^n$, the Euclidean Laplacian,
%\begin{align*}
%	\Delta_0: W^{k+2,p}_{\delta-2}(\mathbb{R}^n)\to W^{k,p}_{\delta}(\mathbb{R}^n)
%\end{align*}
%is Fredholm for  $\delta\in \mathbb{R}\setminus \mathfrak{D}$, where  
%\begin{align*}
%	\mathfrak{D}=\left\{ \delta\in\mathbb{R} \mid -\delta+2-\frac{n}{p}\in \mathbb{N}_0\text{ if }\delta-2\leq -\frac{n}{p} \right\}\cup \left\{ \delta\in\mathbb{R} \mid \delta-\frac{n}{p}\in \mathbb{N}_0\text{ if }\delta-2>- \frac{n}{p} \right\}.
%\end{align*}
%For a proof, see \cite{Lockhart-McOwen}.
%
% Explain the exceptional weights $\delta\in\mathbb{R}$ as [Proposition $6$, p. $508$, \cite{Hau-Hun-Maz}]. 
\subsection{The space of gauged ALE Ricci-flat metrics}\label{sec-moduli-space}
Fix an ALE Ricci-flat  manifold $(M,g_0)$. Let $\mathcal{M}$ be the space of smooth metrics on the manifold $M$. For $g\in \mathcal{M}$, let $V=V(g,g_0)$ be the vector field defined intrinsically by (\ref{def-vect-gauge}) and locally given by $g^{ij}(\Gamma(g)^k_{ij}-\Gamma(g_0)_{ij}^k)$ where $\Gamma(g)_{ij}^k$ denotes the Christoffel symbols associated to the Riemannian metric $g$. We call a metric $g$ gauged, if $V(g,g_0)=0$.
Let
\begin{align}
\mathcal{F}=\left\{g\in\mathcal{M}\mid -2\Ric(g)+\Li_{V(g,g_0)}g=0 \right\}\label{def-F-set},
\end{align}
be the set of stationary points of the Ricci-DeTurck flow. In local coordinates, the above equation (\ref{def-F-set}) can also be written as
\begin{eqnarray*}
	0&=&g^{ab}\nabla^{g_0,2}_{ab}g_{ij}-g^{kl}g_{ip}(g_0)^{pq}\Rm(g_0)_{jklq}-g^{kl}g_{jp}(g_0)^{pq}\Rm(g_0)_{iklq}\\
	&&+g^{ab}g^{pq}\left(\frac{1}{2}\nabla^{g_0}_ig_{pa}\nabla^{g_0}_jg_{qb}+\nabla^{g_0}_ag_{jp}\nabla^{g_0}_qg_{ib}\right)\\
	&&-g^{ab}g^{pq}\left(\nabla^{g_0}_ag_{jp}\nabla^{g_0}_bg_{iq}-\nabla^{g_0}_jg_{pa}\nabla^{g_0}_bg_{iq}-\nabla^{g_0}_ig_{pa}\nabla^{g_0}_bg_{jq}\right),
\end{eqnarray*}
see \cite[Lemma 2.1]{Shi-Def}.
By defining $h=g-g_0$, this equation can  be again rewritten as
\begin{equation}\begin{split}\label{stat-evo-equ}
0&=g^{ab}\nabla^{g_0,2}_{ab}h_{ij}+h_{ab}g^{ka}(g_0)^{lb}g_{ip}(g_0)^{pq}\Rm(g_0)_{jklq}+h_{ab}g^{ka}(g_0)^{lb}g_{jp}(g_0)^{pq}\Rm(g_0)_{iklq}\\
&\quad+g^{ab}g^{pq}\left(\frac{1}{2}\nabla^{g_0}_ih_{pa}\nabla^{g_0}_jh_{qb}+\nabla^{g_0}_ah_{jp}\nabla^{g_0}_qh_{ib}\right)\\&\quad
-g^{ab}g^{pq}\left(\nabla^{g_0}_ah_{jp}\nabla^{g_0}_bh_{iq}-\nabla^{g_0}_jh_{pa}\nabla^{g_0}_bh_{iq}-\nabla^{g_0}_ih_{pa}\nabla^{g_0}_bh_{jq}\right),
\end{split}
\end{equation}
where we used that $g_0$ is Ricci-flat.
The linearization of this equation at $g_0$ is given by
\begin{align*}
\frac{d}{dt}|_{t=0}(-2\mathrm{Ric}_{g_0+th}+\Li_{V(g_0+th,g_0)}(g_0+th))=L_{g_0}h.
\end{align*}
A proof of this fact can be found for instance in \cite[Chapter 3]{Bam-Phd}.

We recall the well-known fact that the $L^2$-kernel of the Lichnerowicz operator consists of transverse traceless tensors:

\begin{lemma}\label{tttensors}
	Let $(M^n,g)$ be an ALE Ricci-flat manifold and $h\in\ker_{L^2}(L_{g_0})$. Then, $\tr_{g_0} h=0$ and $\div_{g_0} h=0$.
\end{lemma}
\begin{proof}
	Straightforward calculations show that $\tr_{g_0}\circ L_{g_0}=\Delta_{g_0}\circ \tr_{g_0}$ and $\div_{g_0}\circ L_{g_0}=\Delta_{g_0}\circ\div_{g_0}$. Therefore, $\tr_{g_0} h\in \ker_{L^2}(\Delta_{g_0})$ and $\div_{g_0} h\in \ker_{L^2} (\Delta_{g_0})$ which implies the statement of the lemma.
\end{proof}
\noindent The next proposition ensures that ALE steady Ricci solitons are Ricci-flat:
\begin{prop}\label{flatsoliton}
	Let $(M^n,g,X)$ be a steady Ricci soliton, i.e. $\Ric(g)=\mathcal{L}_X(g)$ for some vector field $X$ on $M$. Then $\lim_{+\infty}|X|_g=0$ implies $X=0$. In particular, any steady soliton in the sense of (\ref{def-F-set}) that is ALE with $\lim_{+\infty}V(g,g_0)=0$ is Ricci-flat.
	%	\\
	%		\item if $X=\nabla^g f$ for some smooth function being proper and if $\Rm(g)=\textit{O}(r^{-1-\epsilon})$ for some positive $\epsilon$ then the $M$ is diffeomorphic to $\mathbb{R}_+\times F$ outside a compact set where $F$ is a closed flat manifold.
	%	\end{enumerate}
\end{prop}
%\begin{rk}
%	In particular, an asymptotically conical steady Ricci soliton whose asymptotic cone has a section with finite fundamental group cannot be gradient with a proper potential function.
%	This proposition can be applied to get rid of steadies when one studies the stability of asymptotically conical Ricci-flat metrics since the vector field $X$ is likely to be DeTurck's gauge that goes to zero at infinity by construction. In the $L^2-L^{\infty}$ setting, this proposition is unnecessary because one shows by hand the convergence to an asymptotically conical Ricci-flat metric. However, this proposition might be useful for the $L^{\infty}$ case.
%\end{rk}
\begin{proof}
	By the contracted Bianchi identity, one has:
	\begin{eqnarray*}
		\frac{1}{2}\nabla^g\R_g=\div_g\Ric(g)
		=\div_g\mathcal{L}_X(g)
		&=&\frac{1}{2}\nabla^g\tr_g(\mathcal{L}_X(g))+\Delta_gX+\Ric(g)(X)\\
		&=&\frac{1}{2}\nabla^g\R_g+\Delta_gX+\Ric(g)(X).
	\end{eqnarray*}
	Therefore, $\Delta_gX+\Ric(g)(X)=0$. In particular,
	\begin{eqnarray*}
		\Delta_g|X|_g^2+X\cdot|X|_g^2=2|\nabla^gX|_g^2+2<\nabla^g_XX,X>_g-2\Ric(g)(X,X)=2|\nabla^gX|_g^2,
	\end{eqnarray*}
	which establishes that $|X|_g^2$ is a subsolution of $\Delta_X:=\Delta_g+X\cdot$. The use of the maximum principle then implies the result in case $\lim_{+\infty}|X|=0$.
\end{proof}
\begin{theo}\label{ell-reg-prop}
	Let $(M^n,g_0)$ be an ALE Ricci-flat manifold  with order $\tau>0$.
	Let $g\in\mathcal{F}$ be in a sufficiently small neighbourhood of $g_0$ with respect to the
	$L^2\cap L^{\infty}$-topology.
	Then $g$ is an ALE Ricci-flat manifold of order $n-1$ with respect to the same coordinates as $g_0$.
	%	\textcolor{blue}{In the ALF case: Let $(M^{n+1},g_0)$ be a Ricci-flat manifold of order $\kappa>0$. Then,
	%		$(M^{n+1},g)$ is a Ricci-flat ALF manifold of the same order as $g_0$. We can even say more:
	%		$\nabla^{g_0,k}(g-g_0)=\textit{O}(r^{-\alpha-k})$ for any $\alpha$ smaller than
	%		\begin{align*}
	%		n-2+\min\left\{\kappa,1\right\}\quad \text{if }n\geq 4,\qquad \frac{3}{2}+\min\left\{\frac{1}{2},\kappa\right\}\quad  \text{if }n=3
	%		\end{align*}
	%(Recall that $\kappa\leq n-2$ for all nontrivial ALF metrics due to Theorem 1 in Minerbe, a mass for ALF manifolds. )}
\end{theo}
%\textcolor{blue}{
%\begin{coro}
%	If $M$ is ALF, the mass defined in Minerbe is constant for all metrics in $\mathcal{F}$.
%\end{coro}}
\begin{rk}	
	\begin{itemize}
		%		\item[(i)]The assumptions on $(M^n,g_0)$ imply by \cite{Ban-Kas-Nak} that $\tau=n-1$ in general and $\tau=n$ if $n=4$ or if $g$ is K\"ahler.
		\item[(i)] If $n=4$ or $g_0$ is K\"ahler, it seems likely that $g\in\mathcal{F}$ is ALE  of order $n$ with respect to the same coordinates as $g_0$. However, we don't need this decay for further considerations.
		\item[(ii)] A priori, Proposition \ref{ell-reg-prop} does not assume any integral or pointwise decay on the difference tensor $g-g_0$ or on the curvature tensor of $g$.
		The assumptions on $g$ can be even weakened as follows: If $\left\|g-g_0\right\|_{L^p(g_0)}\leq K<\infty$ for some $p\in[2,\infty)$ and $\left\|g-g_0\right\|_{L^{\infty}(g_0)}<\epsilon=\epsilon(g_0,p,K)$, then the conclusions of Theorem \ref{ell-reg-prop} hold.		
	\end{itemize}
\end{rk}
\begin{proof}[Proof of Theorem \ref{ell-reg-prop}]
	%	Let us consider the ALE case first.
	The first step consists in applying a Moser iteration to the norm of the difference of the two metrics : $\arrowvert h\arrowvert_{g_0}:=|g-g_0|_{g_0}$. Indeed, recall that $h$ satisfies \eqref{stat-evo-equ} which can also be written as
	\begin{eqnarray*}
		&&g^{-1}\ast\nabla^{g_0,2}h+2\Rm(g_0)\ast h=\nabla^{g_0}h\ast\nabla^{g_0}h,\\
		&& g^{-1}\ast \nabla^{g_0,2}h_{ij}:=g^{kl}\nabla^{g_0,2}_{kl}h_{ij}.
	\end{eqnarray*}
	In particular,
	\begin{eqnarray*}
		g^{-1}\ast\nabla^{g_0,2}\arrowvert h\arrowvert_{g_0}^2&=&2g^{-1}(\nabla^{g_0}h,\nabla^{g_0}h)-4\langle\Rm(g_0)\ast h,h\rangle_{g_0}+\langle\nabla^{g_0}h\ast \nabla^{g_0}h,h\rangle_{g_0},\\
		g^{-1}(\nabla^{g_0}h,\nabla^{g_0}h)&:=&g^{ij}\nabla^{g_0}_ih_{kl}\nabla^{g_0}_jh_{kl}.
	\end{eqnarray*}
	Therefore, as $\|h\|_{L^{\infty}(g_0)}\leq \epsilon$ where $\epsilon>0$ is a sufficiently small constant depending on $n$ and $g_0$, we get
	\begin{align*}
		g^{-1}\ast\nabla^{g_0,2}\arrowvert h\arrowvert_{g_0}^2&\geq 2|\nabla^{g_0}h|_{g_0}^2+2(g^{-1}-g_0^{-1})(\nabla^{g_0}h,\nabla^{g_0}h)\\
		&\quad-c(n)\arrowvert\Rm(g_0)\arrowvert_{g_0}\arrowvert h\arrowvert^2_{g_0}-
		c(n)|h|_{g_0}|\nabla^{g_0}h|_{g_0}^2\\
	&	\geq	-c(n)\arrowvert\Rm(g_0)\arrowvert_{g_0}\arrowvert h\arrowvert^2_{g_0}.
	\end{align*}
	As $|\Rm(g_0)|\in L^{n/2}(M)$ and $h\in L^2(S^2T^*M)$, Lemma $4.6$ and Proposition $4.8$ of \cite{Ban-Kas-Nak} tell us that $|h|^2=\textit{O}(r^{-\tau})$ at infinity for any positive $\tau<n-2$, i.e. $h=\textit{O}(r^{-\tau})$ for any $\tau<n/2-1$. Here, $r$ denotes the distance function on $M$ centered at some arbitrary point $x\in M$.
	
	The next step is to show that
	$\nabla^{g_0} h=\textit{O}(r^{-\tau-1})$ for $\tau<n/2-1$. Assume $p\geq2$.
    We first proceed with a chain of inequalities as follows:
	\begin{align*}
	\left\|h\right\|_{W^{2,p}_{-\tau}(g_0)}&\leq C(\left\|g^{-1}\ast\nabla^{g_0,2}h\right\|_{L^p_{-\tau-2}(g_0)}+\left\|h\right\|_{L^p_{-\tau}(g_0)})\\
	%\leq& C(\left\|R_0[h]\right\|_{L^p_{-\tau}}+\left\|\nabla^{g_0}R_1[h]\right\|_{L^p_{-\tau-2}}+\left\|h\right\|_{L^p_{-\tau}})\\
	&\leq C(\left\| |\nabla^{g_0} h|^2\right\|_{L^{p}_{-\tau-2}(g_0)}+\left\|h\right\|_{L^p_{-\tau}(g_0)})\\
	&\leq C( \left\|\nabla^{g_0,2}h\right\|_{L^p_{-\tau-2}(g_0)}+\left\|\nabla^{g_0} h\right\|_{L^p_{-\tau-1}(g_0)})\left\|h\right\|_{L^{\infty}(g_0)}+C\left\|h\right\|_{L^p_{-\tau}(g_0)}.
	\end{align*}
	Here, the first inequality follows from elliptic regularity for weighted Sobolev spaces, see \cite[Theorem $4.21$]{Mar-Phd}. The second inequality uses equation \eqref{stat-evo-equ} and the third inequality follows from an interpolation inequality for weighted spaces, see Lemma \ref{weighted_interpolation} below.
	This implies
	\begin{align*}
	\left\|h\right\|_{C^{1,\alpha}_{-\tau}(g_0)}\leq C\left\|h\right\|_{W^{2,p}_{-\tau}(g_0)}\leq C \left\| h\right\|_{L^p_{-\tau}(g_0)}\leq C \left\| h\right\|_{L^{\infty}_{-\tau-\epsilon}(g_0)}<\infty
	\end{align*}
	for all $p\in(n,\infty)$ and $\tau+\epsilon<\frac{n}{2}-1$ provided that the $L^{\infty}$-norm is small enough.
Here, the first inequality follows from weighted Sobolev embedding, see \cite[Remark 6.9]{Pac-T}. The second inequality follows from the above and the third inequality follows from the weighted H\"{o}lder inequality (see e.g.\ \cite[Lemma 6.7]{Pac-T}) applied to $h$ and the weight function itself.	
	 Consequently,
	\begin{eqnarray*}
		\nabla^{g_0} h=\textit{O}(r^{-1-\tau}).
	\end{eqnarray*}
	In the following we will further improve the decay order and show that $h=\textit{O}(r^{-n+1})$. As a consequence of elliptic regularity for weighted H\"{o}lder spaces (\cite[Theorem $4.21$]{Mar-Phd}), we will furthermore get $\nabla^{g_0,k}h=\textit{O}(r^{-n+1-k})$ for any $k\in\mathbb{N}$. To prove these statements we adapt the strategy in \cite[p. 325-327]{Ban-Kas-Nak}.
	
	As $\Rm(g_0)=\textit{O}(r^{-n-1})$, $h=\textit{O}(r^{-\tau})$ and $\nabla^{g_0} h=\textit{O}(r^{-\tau-1})$ for some fixed $\tau$ slightly smaller than $\frac{n}{2}-1$, equation \eqref{stat-evo-equ} implies that $\Delta_{g_0} h=\textit{O}(r^{-2\tau-2})$. Thus, $\Delta_{g_0} h=\textit{O}(r^{-\mu})$ for some $\mu$ slightly smaller than $n$.
	
	Let $\varphi: M\setminus B_R(x)\to (\mathbb{R}^n\setminus B_R)/\Gamma $ be coordinates at infinity with respect to which $g_0$ is ALE of order $n-1$. Let furthermore $\Pi:\mathbb{R}^n\setminus B_R\to (\mathbb{R}^n\setminus B_R)/\Gamma$ be the projection map. From now on,
	we consider the objects on $M$ as objects on $\mathbb{R}^n\setminus B_R$ by identifying them with the pullbacks under the map $\Pi\circ \varphi^{-1}$. To avoid any confusion, we denote the pullback of $\Delta_{g_0}$ by $\Delta$. The operator $\Delta_0$ denotes the Euclidean Laplacian on $\mathbb{R}^n\setminus B_R$.
	
	Let $r_0=|z|$ be the euclidean norm as a function on  $\mathbb{R}^n\setminus B_R$. Then we have, for any $\beta\in\mathbb{R}$,
	\begin{align*}
	\Delta r_0^{-\beta}=\Delta_0r_0^{-\beta}+(\Delta-\Delta_0)r_0^{-\beta}=\beta(\beta-n+2)r_0^{-\beta-2}+O(r_0^{-\beta-n-1}).
	\end{align*}
	Let $u=h_{ij}$ for any $i,j$. For any constant $A>0$, we have
	\begin{align*}
	\Delta (A\cdot r_0^{-\beta}\pm u)= (A\cdot\beta(\beta-n+2)+ \textit{O}( r_0^{-n+1})+\textit{O}(r_0^{-\mu+\beta+2}))r_0^{-\beta-2}.
	\end{align*}
	If we choose $\beta>0$ so that $\beta+2<\mu<n$, we can choose $A$ so large that
	\begin{align*}
	\Delta (A\cdot r_0^{-\beta}\pm u)<0 \qquad A\cdot r_0^{-\beta}\pm u>0 \text{ if }r_0=R.
	\end{align*}
	The strong maximum principle then implies that $u=\textit{O}(r_0^{-\beta})$ and we also get $\Delta u=\textit{O}(r_0^{-\beta-2})$. By elliptic regularity, $u\in W^{2,p}_{-\beta}(\mathbb{R}^n\setminus B_R)$ for all $p\in[1,\infty)$.
	 We now claim that we can \eqref{stat-evo-equ} to prove $\Delta_0 u\in L^{p}_{-2\beta-2}(\mathbb{R}^n\setminus B_R)$. To do so, we first compute
	 \begin{align*}
	 g^{ab}\nabla^{g_0,2}_{ab}h_{ij}&=\Delta_0h_{ij}+(g^{ab}-g_0^{ab})\nabla^{g_0,2}_{ab}h_{ij}+(g_0^{ab}-{\delta}^{ab})\nabla^{g_0,2}_{ab}h_{ij}\\
	 &\qquad +\delta^{ab}(\Gamma(g_0)*\Gamma(g_0)*h+\partial\Gamma(g_0)*h+\Gamma(g_0)*\nabla^{g_0} h).
	 \end{align*}
	 Here $\partial$ denotes the coordinate derivative on $\mathbb{R}^n$. We rewrite \eqref{stat-evo-equ} schematically as
	 \begin{align*}
	 g^{ab}\nabla^{g_0,2}_{ab}h=\Rm(g_0)*h+\nabla^{g_0}h*\nabla^{g_0}h.
	 \end{align*}
	 A combination of these two formulas combined with standard estimates yields
	 \begin{align*}
	 \left\|\Delta_0 u\right\|_{L^{p}_{-2\beta-2}}&\leq C(\left\|h\right\|_{L^{\infty}_{-\beta}}\left\|\nabla^{g_0,2}h\right\|_{L^p_{-\beta-2}}+\left\|g_0-\delta\right\|_{L^{\infty}_{-\beta}}\left\|\nabla^{g_0,2}h\right\|_{L^p_{-\beta-2}}\\
	 &\qquad+\left\|\Rm(g_0)\right\|_{L^{\infty}_{-\beta-2}}\left\|h\right\|_{L^p_{\beta}}+\left\|\nabla^{g_0}h\right\|_{L^{2p}_{-\beta-1}}^2\\
	 &\qquad +\left\|\Gamma(g_0)*\Gamma(g_0)\right\|_{L^{\infty}_{-\beta-2}}\left\|h\right\|_{L^p_{-\beta}}+\left\|\partial\Gamma(g_0)\right\|_{L^{\infty}_{-\beta-2}}\left\|h\right\|_{L^p_{-\beta}}\\
	 &\qquad +\left\|\Gamma(g_0)\right\|_{L^{\infty}_{-\beta-1}}\left\|\nabla^{g_0}h\right\|_{L^p_{-\beta-1}}
	 )
	 \end{align*}
	 and the right hand side is finite because $g_0-\delta=O(r^{-n+1})$, $\Gamma(g_0)=O(r^{-n})$, $\partial\Gamma(g_0)=O(r^{-n-1})$, $\Rm(g_0)=O(r^{-n-1})$ and $\beta<n-2$. This proves the claim.
	 
	 Recall that a weight parameter $\delta\in\mathbb{R}$ is called exceptional if $k\notin \mathbb{Z}\setminus\left\{-1,-2,\ldots,3-n\right\}$. Now for all non-exceptional values $-\gamma$, \cite[Theorem 1.7]{Bar-Mass} 
	implies that the Euclidean Laplacian is an isomorphism as a map  \begin{align*}\Delta_0:W'^{2,p}_{-\gamma}(\mathbb{R}^n\setminus \left\{0 \right\})\to L'^p_{-\gamma-2}(\mathbb{R}^n\setminus \left\{0 \right\}).
	\end{align*}
	Thus for all non-exceptional values $-\gamma$ satisfying  $\gamma<2\beta<2(n-2)$, there exists a function $v^{\gamma}_{ij}\in W^{2,p}_{-\gamma}(\mathbb{R}^n\setminus B_R)$ such that $\Delta_0v^{\gamma}_{ij}=\Delta_0h_{ij}$. 
	By the expansion of harmonic functions on $\mathbb{R}^n$, we have
	\begin{align*}
	h_{ij}-v^{\gamma}_{ij}=A_{ij}r_0^{-n+2}+\textit{O}(r_0^{-n+1}),
\qquad \partial_ih_{jk}-\partial_iv^{\gamma}_{jk}=(2-n)|z|^{-n}A_{jk}z_i+	\textit{O}(r_0^{-n}).
	\end{align*}
	Here, $z=(z_1,\ldots,z_n)$ denotes the coordinates on $\mathbb{R}^n$ so that $|z|=r_0$.
	Sobolev ebedding implies $v^{\gamma}_{ij}\in C^{1,\alpha}_{1-n}(\mathbb{R}^n\setminus B_R)$. Therefore,
\begin{align}\label{first_decay}
h_{ij}=A_{ij}r_0^{-n+2}+\textit{O}(r_0^{-n+1}),\qquad  \partial_ih_{jk}=(2-n)|z|^{-n}A_{jk}z_i+\textit{O}(r_0^{-n}).
\end{align}	
	 We now can use Proposition \ref{flatsoliton} to improve this decay rate slightly by getting rid of the $A_{ij}$.
	In fact, the proposition implies that the equations $\mathrm{Ric}(g)=0$ and $V(g,g_0)=0$ hold individually. Therefore, \begin{equation}\begin{split}\label{improved_decay}0=g^{ij}(\Gamma(g)_{ij}^k-\Gamma(g_0)_{ij}^k)&=
\frac{1}{2}g^{ij}g^{kl}(\partial_ig_{jl}+\partial_jg_{il}-\partial_lg_{ij})\\&\qquad -\frac{1}{2}g^{ij}(g_0)^{kl}(\partial_i(g_0)_{jl}+\partial_j(g_0)_{il}-\partial_l(g_0)_{ij})	\\
	&=\frac{1}{2}\delta^{ij}\delta^{kl}(\partial_ih_{jl}+\partial_jh_{il}-\partial_lh_{ij})\\
	&\qquad +\frac{1}{2}(g^{ij}-\delta^{ij})\delta^{kl}(\partial_ih_{jl}+\partial_jh_{il}-\partial_lh_{ij})\\
	&\qquad +\frac{1}{2}g^{ij}(g^{kl}-\delta^{kl})(\partial_ih_{jl}+\partial_jh_{il}-\partial_lh_{ij})
		\\
	&\qquad+\frac{1}{2}g^{ij}(g^{kl}-(g_0)^{kl})(\partial_i(g_0)_{jl}+\partial_j(g_0)_{il}-\partial_l(g_0)_{ij})\\
	&= \frac{1}{2}\delta^{ij}\delta^{kl}(\partial_ih_{jl}+\partial_jh_{il}-\partial_lh_{ij})\\
	&\qquad+\textit{O}(r_0^{-n+2})\cdot \textit{O}(r_0^{-n+2})+\textit{O}(r_0^{-n+2})\cdot \textit{O}(r_0^{-n+1})
	\\ &=
	(Az-\frac{1}{2}\tr(A)z)^k(2-n)|z|^{-n}+\textit{O}(r_0^{-n}),
	\end{split}
	\end{equation}
	which implies that $A_{ij}=0$ and thus, $h=\textit{O}(r^{-n+1})$.
	As $h_{ij}=v^{\gamma}_{ij}+\textit{O}(r_0^{-n+1})$ with $v^{\gamma}_{ij}\in W^{2,p}_{-\gamma}(\mathbb{R}^n\setminus B_R)$ and an harmonic remainder term, Sobolev embedding and elliptic regularity imply that $h_{ij}\in C^{1,\alpha}_{1-n}(\mathbb{R}^n\setminus B_R)$, so that $\nabla^{g_0} h=\textit{O}(r^{-n})$.
	Elliptic regularity for weighted H\"{o}lder spaces (\cite[Theorem $4.21$]{Mar-Phd} again) implies that $\nabla^{g_0,k} h=\textit{O}(r^{-n+1-k})$ for all $k\in\mathbb{N}$.
\end{proof}
\begin{lemma}\label{weighted_interpolation}
Let $p\in [2,\infty)$, $\tau\in\mathbb{R}$ and a fixed ALE-metric be given. Let $h$ be a section in a Riemannian vector bundle  with connection.
If $h\in L^{\infty}$ with $\nabla h\in W^{1,p}_{\tau+1}$, we have $|\nabla h|^2\in L^p_{\tau}$ and the weighted interpolation inequality
	\begin{align*}
	 \left\| |\nabla h|^2\right\|_{L^{p}_{\tau}}
	&\leq C( \left\|\nabla^{2}h\right\|_{L^p_{\tau}}+\left\|\nabla h\right\|_{L^p_{\tau+1}})\left\|h\right\|_{L^{\infty}}.
	\end{align*}
\end{lemma}
\begin{proof}
By density, we may assume that $h$ is compactly supported. To avoid differentiability issues at $0$, we introduce the quantity $|h|_{\delta}=(\langle h,h\rangle +\delta)^{1/2}$. Recall that $\rho$ is the weight function on the manifold. Then we compute
\begin{align*}
\int_M |\nabla h|^{2p}\rho^{\alpha}d\mu&\leq \int_M \langle \nabla h,\nabla h\rangle |\nabla h|_{\delta}^{2p-2}\rho^{\alpha}d\mu\\
&=-\int_M \langle h,\Delta h\rangle|\nabla h|_{\delta}^{2p-2}\rho^{\alpha}d\mu+\int_M \langle h,\nabla h\rangle  (\nabla  |\nabla h|_{\delta}^{2p-2})\rho^{\alpha}d\mu\\
&\qquad + \int_M \langle h,\nabla h\rangle  |\nabla h|_{\delta}^{2p-2}\nabla\rho\cdot\rho^{\alpha-1}d\mu\\
&\leq \int_M |h||\nabla^2h||\nabla h|_{\delta}^{2p-2}\rho^{\alpha}d\mu+C\int_M |h||\nabla h|^2|\nabla^2h||\nabla h|_{\delta}^{2p-4}\rho^{\alpha}d\mu\\
&\qquad +C\int_M |h||\nabla h|  |\nabla h|_{\delta}^{2p-2}\rho^{\alpha-1}d\mu,
\end{align*}
where we used that $|\nabla \rho|$ is bounded in the last step. By letting $\delta\to 0$, we obtain
\begin{align*}
\int_M |\nabla h|^{2p}\rho^{\alpha}s\mu&\leq C\int_M |h||\nabla^2h||\nabla h|^{2p-2}\rho^{\alpha}d\mu+C\int_M |h||\nabla h|^{2p-1}\rho^{\alpha-1}d\mu.
\end{align*}
By the Young inequality,
\begin{align*}
|h||\nabla^2h||\nabla h|^{2p-2}&\leq \epsilon |\nabla h|^{2p}+C(\epsilon)|h|^p|\nabla^2 h|^p,\\
|h||\nabla h|\rho^{-1}|\nabla h|^{2p-2}&\leq \epsilon |\nabla h|^{2p}+C(\epsilon)|h|^p|\nabla h|^p\rho^{-p}.
\end{align*}
Therefore,
\begin{align*}
\int_M |\nabla h|^{2p}\rho^{\alpha}d\mu&\leq 2C\epsilon \int_M |\nabla h|^{2p} \rho^{\alpha}d\mu+ C\cdot C(\epsilon)
\int_M |h|^p|\nabla^2 h|^p\rho^{\alpha}d\mu\\ &\qquad +C\cdot C(\epsilon)\int_M |h|^p|\nabla h|^p\rho^{\alpha-p}d\mu.
\end{align*}
Let us choose $\epsilon<\frac{1}{2C}$. Then we obtain, by subtracting the first term on the right hand side and by dividing by a constant
\begin{align*}
\int_M |\nabla h|^{2p}\rho^{\alpha}d\mu\leq C
\int_M |h|^p|\nabla^2 h|^p\rho^{\alpha}d\mu +C\int_M |h|^p|\nabla h|^p\rho^{\alpha-p}d\mu.
\end{align*}
Taking the $p$-th root yields, with a new constant $C$:
\begin{align*}
\left(\int_M |\nabla h|^{2p}\rho^{\alpha}d\mu\right)^{\frac{1}{p}}&\leq C\left(\int_M |h|^p|\nabla^2 h|^p\rho^{\alpha}d\mu\right)^{\frac{1}{p}}+C\left( \int_M |h|^p|\nabla h|^p\rho^{\alpha-p}d\mu\right)^{\frac{1}{p}}\\
&\leq C\left\|h\right\|_{L^{\infty}}\left\{\left(\int_M |h|^p|\nabla^2 h|^p\rho^{\alpha}d\mu\right)^{\frac{1}{p}}+\left( \int_M |h|^p|\nabla h|^p\rho^{\alpha-p}d\mu\right)^{\frac{1}{p}} \right\}
\end{align*}
Inserting $\alpha=-\tau p-n$ yields the desired inequality.
\end{proof}

\begin{rk}\label{lin-decay}
	The proof of the above theorem also applies to $h\in\ker_{L^2}(L_{g_0})$ with the only difference that the first equality in 
	\eqref{improved_decay} is replaced by the
	formula $0=\div_{g_0}h-\frac{1}{2}\nabla^{g_0}\tr_{g_0} h$ but which admits the same asymptotic expansion. This formula in turn  holds because $h$ is a TT-tensor due to Proposition \ref{tttensors}.
	Therefore we can conclude that $h=\textit{O}_{\infty}(r^{-n+1})$  if $h\in\ker_{L^2}(L_{g_0})$.
\end{rk}
%\begin{rk}\label{cont-dep}
%	Note that by the above proof and by the estimates in \cite[Section 4]{Ban-Kas-Nak}, the $C^{k,\alpha}_{1-n}$-norm of $g-g_0$ for $g\in\mathcal{F}$ is continuous with respect to the $L^p\cap L^{\infty}$-topology.
%\end{rk}
\begin{theo}\label{analyticset}
	Let $(M,g_0)$ be an ALE Ricci-flat metric and $\mathcal{F}$ as above. Then there exists an $L^2\cap L^{\infty}$-neighbourhood $\mathcal{U}$ of $g_0$ in the space of metrics and a finite-dimensional real-analytic submanifold $\mathcal{Z}\subset\mathcal{U}$ with $T_{g_0}\mathcal{Z}=ker_{L^2}(L_{g_0})$
	such that $\mathcal{U}\cap\mathcal{F}$
	is an analytic subset of $\mathcal{Z}$. In particular if $g_0$ is integrable, we have  $\mathcal{U}\cap\mathcal{F}=\mathcal{Z}$.	
\end{theo}
\begin{proof}
	Let $\Phi:g\mapsto-2\Ric(g)+\Li_{V(g,g_0)}(g)$ and $\mathcal{B}_{L^2\cap L^{\infty}}(g_0,\epsilon)$ be the $\epsilon$-ball with respect to the $L^2\cap L^{\infty}$-norm induced by $g_0$ and centered at $g_0$.
	By Theorem \ref{ell-reg-prop}, we can choose $\epsilon>0$ be so small  that any $g\in\mathcal{F}\cap \mathcal{B}_{L^2\cap L^{\infty}}(g_0,\epsilon)$ satisfies the condition $g-g_0=\textit{O}_{\infty}(r^{-n+1}$) so that $\left\|g-g_0\right\|_{H^k_{\delta}}<\infty$ for any $k\in \mathbb{N}$ and $\delta >-n+1$.
	
	Suppose now in addition that $k>n/2+2$ and $\delta\leq-n/2$ 
	and let $\mathcal{V}$ be a $H^k_{\delta}$-neighbourhood of $g_0$ with 
	$\mathcal{V}\subset \mathcal{B}_{L^2\cap L^{\infty}}(g_0,\epsilon_1)$. Then the map 
	$\Phi$, considered as a map
	$\Phi: H^k_{\delta}(S^2T^*M)\supset\mathcal{V}\to H^{k-2}_{\delta-2}(S^2T^*M)$ is a real analytic map between Hilbert manifolds. If $\delta$ is nonexceptional, the differential $d\Phi_{g_0}=L_{g_0}: H^k_{\delta}(S^2T^*M)\to H^{k-2}_{\delta-2}(S^2T^*M)$ is Fredholm.
	By  \cite[Lemma $13.6$]{Koiso-Cx-Str}, there exists (possibly after passing to a smaller neighbourhood) a finite-dimensional real-analytic submanifold
	$\mathcal{W}\subset \mathcal{V}$ with $g_0\in \mathcal{W}$ and $T_{g_0}\mathcal{W}=\ker_{H^k_{\delta}}(L_{g_0})$
	such that $\mathcal{V}\cap\Phi^{-1}(0)\subset\mathcal{W}$ is a real-analytic subset.
	
	By the proof of Theorem \ref{ell-reg-prop}, we can choose an $L^2\cap L^{\infty}$-neighbourhood  $\mathcal{U}_{g_0}\subset  \mathcal{B}_{L^2\cap L^{\infty}}(g_0,\epsilon)$  of $g_0$ so small that $\mathcal{U}_{g_0}\cap\mathcal{F}\subset\mathcal{V}$ (provided that $\mathcal{V}$ is small enough). Then the set $\mathcal{Z}=\mathcal{U}_{g_0}\cap\mathcal{W}$ fulfills the desired properties because
	$T_{g_0}\mathcal{Z}=T_{g_0}\mathcal{W}=\ker_{H^k_{\delta}}(L_{g_0})=\ker_{L^2}(L_{g_0})$ due to the asymptotics of elements in $\ker_{L^2}(L_{g_0})$ shown in Proposition \ref{ell-reg-prop} and Remark \ref{lin-decay}.
	%	
	%	
	%	By the proof of Theorem \ref{ell-reg-prop}, $h=g-g_0=\textit{O}_{\infty}(r^{-n+1}$) and,
	%	$\left\|h\right\|_{H^k_{\delta}}<\epsilon$ for all $k\in \mathbb{N}$ and $\delta >-n+1$ provided that $g\in\mathcal{U}$ with $\mathcal{U}$ small enough.
	%	Because of Remark \ref{cont-dep} and since $\left\|h\right\|_{L^2}+\left\|h\right\|_{L^{\infty}}\leq C\left\|h\right\|_{H^k_{\delta}} $ for $\delta\in(-n+1,0)$ and $k>n/2$, there exists a $H^k_{\delta}$-neighbourhood $\mathcal{V}$ of $g_0$ such that
	%	\begin{align*}
	%	\mathcal{U}\cap\widetilde{\mathcal{F}}=\mathcal{V}\cap\widetilde{\mathcal{F}}.
	%	\end{align*}
	%	From now on let  $\delta\in(-n+1,0)$ and $k>n/2+2$. Then $\Phi$ as a map
	%	$\Phi: H^k_{\delta}(S^2T^*M)\supset\mathcal{V}\to H^{k-2}_{\delta-2}(S^2T^*M)$ is a real analytic map between Hilbert manifolds. If $\delta$ is nonexceptional, the differential $d\Phi_{g_0}=L_{g_0}: H^k_{\delta}(S^2T^*M)\to H^{k-2}_{\delta-2}(S^2T^*M)$ is Fredholm.
	%    The assertion now follows from Lemma 13.6 in the appendix of Koiso, Einstein metrics and complex structures, after possibly passing to smaller neighbourhoods. Observe also that by
	%    the second part of Theorem \ref{ell-reg-prop}, $\ker_{L^2}(L_{g_0})=\ker_{H^k_{\delta}}(L_{g_0})$.
\end{proof}

\begin{prop}\label{slice}
	Let $(M^n,g_0)$ be an ALE Ricci-flat manifold and let $k>n/2+1$ and $\delta\in(-n+1, -n/2]$ nonexceptional. 
	Then there exists a $H^k_{\delta}$-neighbourhood $\mathcal{U}^k_{\delta}$ of $g_0$ in the space of metrics such that
	the set
	\begin{align*}
	\mathcal{G}_{\delta}^k:=\left\{g\in\mathcal{U}_{\delta}^k\mid g^{ij}(\Gamma(g)_{ij}^l-\Gamma(g_0)_{ij}^l)=0\right\},
	\end{align*}
	is a smooth manifold. Moreover, for any $g\in \mathcal{U}_{\delta}^k$, there exists a unique diffeomorphism $\varphi$ which is $H^{k+1}_{\delta+1}$-close to the identity such that $\varphi^*g\in \mathcal{G}_{\delta}^k$.
	%	\begin{align*}
	%	\varphi^*g\in \mathcal{H}_{\delta}^k:=\left\{g\in\mathcal{U}_{\delta}^k\mid g^{ij}(\Gamma_{ij}^k-\bar{\Gamma}_{ij}^k)=0\right\}
	%	\end{align*}
	%	where .
\end{prop}
\begin{proof}Let $\mathcal{U}$ be a $H^{k}_{\delta}$-neighbourhood of $g_0$ in the space of metrics such that the map $V:H^k_{\delta}(S^2T^*M)\supset\mathcal{U}\to H^{k-1}_{\delta-1}(TM)$, given by $V(g)^l=V(g,g_0)^l=g^{ij}(\Gamma(g)_{ij}^l-\Gamma(g_0)_{ij}^l)$ is well-defined.
	Linearization  at $g_0$ yields the map $F:H^k_{\delta}(S^2T^*M)\to H^{k-1}_{\delta-1}(TM)$, defined by $F(h)=(\mathrm{div}_{g_0}h)^{\sharp}-\frac{1}{2}\nabla^{g_0}\mathrm{tr}_{g_0}h$. To prove the theorem, it suffices to prove that $F$ is surjective and that the decomposition
	\begin{align*}
	H^k_{\delta}(S^2T^*M)=\ker F\oplus \Li(g_0)(H^{k+1}_{\delta+1}(TM))
	\end{align*}
	holds (here, $\Li_X(g_0)$ denotes the Lie-Derivative of $g_0$ along $X$).
	In fact, a calculation shows that $F\circ \Li(g_0)= \Delta_{g_0}+\Ric(g_0)(\cdot)=\Delta_{g_0}$ since $g_0$ is Ricci-flat. Since the map $\Delta_{g_0}:H^{k+1}_{\delta+1}(\Lambda^1M)\to  H^{k-1}_{\delta-1}(\Lambda^1M)$ is an isomorphism, it follows that $F$ is surjective and $\ker F\cap \Li(g_0)(H^{k+1}_{\delta+1}(TM))=\left\{0\right\}$.
	To show that
	\begin{align*}
	H^k_{\delta}(S^2T^*M)\subset \ker_{L^2}(F)\oplus \Li(g_0)(H^{k+1}_{\delta+1}(TM)),
	\end{align*}  
	let $h\in  H^k_{\delta}(S^2T^*M)$ and $X\in H^{k+1}_{\delta+1}(TM)$ the unique solution
	of $F(h)=\Delta_{g_0} X=F( \Li_X(g_0))$. Then, $h=(h-\Li_X(g_0))+\Li_X(g_0)$ is the desired decomposition.
	By surjectivity of $F$, $\mathcal{G}_{\delta}^k$ is a manifold. The second assertion follows because the map
	\begin{align*}
	\Phi:  \mathcal{G}_{\delta}^k\times H^{k+1}_{\delta+1}(\mathrm{Diff}(M))\to \mathcal{M}_{\delta}^k=:\mathcal{M}\cap H^k_{\delta}(S^2T^*M),\qquad
	(g,\varphi)\mapsto \varphi^*g
	\end{align*}
	is a local diffeomorphism around $g_0$ due to the implicit function theorem and the above decomposition.
\end{proof}
\begin{rk}The construction in Proposition \ref{slice} is similar to the slice provided by Ebin's slice theorem \cite{Ebi-Slice} in the compact case. The set $\mathcal{F}$ is similar to the local premoduli space of Einstein metrics defined in \cite[Definition 2.8]{Koiso-Cx-Str}. In contrast to the compact case, the elements in $\mathcal{F}$ close to $g_0$ can all be homothetic. In fact, this holds for the Eguchi-Hanson metric, see \cite{Page-K3}. More generally, any four-dimensional ALE hyperk\"{a}hler manifold $(M,g)$ admits a three-dimensional subspace of homothetic metrics in $\mathcal{F}$: see \cite[p. 52--53]{Via-Lect-IAS}.
\end{rk}
\subsection{ ALE Ricci-flat K\"{a}hler  spaces}\label{Ricci-Flat-ALE-Kahler}
\begin{lemma}[$\partial\bar{\partial}$-Lemma for ALE manifolds]
	Let $(M,g,J)$ be an ALE K\"{a}hler  manifold, $\delta+2\leq-n/2+2$ nonexceptional, $k\geq 1$ and $\alpha\in H^k_{\delta}(\Lambda^{p,q}M)$. Suppose that
	\begin{itemize}
		\item $\alpha=\partial\beta$ for some $\beta\in H^{k+1}_{\delta+1}(\Lambda^{p-1,q}M)$ and $\bar{\partial}\alpha=0$ or
		\item $\alpha=\bar{\partial}\beta$ for some $\beta\in H^{k+1}_{\delta+1}(\Lambda^{p,q-1}M)$ and ${\partial}\alpha=0$.
	\end{itemize}
	Then there exists a form $\gamma\in H^{k+2}_{\delta+2}(\Lambda^{p-1,q-1}M)$ such that $\alpha=\partial\bar{\partial}\gamma$. Moreover, we can choose $\gamma$ to satisfy the estimate $\left\|\gamma\right\|_{H^{k+2}_{\delta+2}}\leq C\cdot \left\|\alpha\right\|_{H^{k}_{\delta}}$, for some $C>0$.			
\end{lemma}
\begin{proof}This follows along the lines of  Lemma 5.50 in \cite{Ballmann-Book}, except that we have to replace the standard Sobolev spaces by weighted ones to capture the ALE condition.
	Let $d=\partial$ or $d=\bar{\partial}$ and $\Delta=\Delta_{\partial}=\Delta_{\bar{\partial}}$
	Consider $\Delta$ as an operator $\Delta:H^{k+2}_{\delta+2}(\Lambda^{*}M)\to H^{k}_{\delta}(\Lambda^{*}M)$.
	Because of the assumption on $\delta$, it is Fredholm and we have the $L^2$-orthogonal decomposition
	\begin{align*}
	H^k_{\delta}(\Lambda^{*}M)=\ker_{L^2}(\Delta)\oplus \Delta(H^{k+2}_{\delta+2}(\Lambda^{*}M)),
	\end{align*}
	for $\delta\in [-n+1,-n/2)$, \cite[Theorem 8.4.1]{Joy-Hol-Boo}. Moreover, Fredholm properties of elliptic operators on ALE manifolds (see e.g. \cite[Section 10]{Pac-T}) imply that $\Delta$ is a Fredholm map of index zero on these spaces.
Thus, we can define the Green's operator $G$ which is zero on $\ker_{L^2}(\Delta)$ and the inverse of $\Delta$ on $\ker_{L^2}(\Delta)^{\perp}$.
	This defines a continuous linear operator $G:H^{k}_{\delta}(\Lambda^{*}M)\to H^{k+2}_{\delta+2}(\Lambda^{*}M)$.
	By Hodge theory and because $d+d^* :H^{k+1}_{\delta+1}(\Lambda^{*}M)\to H^{k}_{\delta}(\Lambda^{*}M)$ is also Fredholm,
	\begin{align*}
	d( H^{k+1}_{\delta+1}(\Lambda^{*}M))\oplus d^{*}(H^{k+1}_{\delta+1}(\Lambda^{*}M))=\Delta(H^{k+2}_{\delta+2}(\Lambda^{*}M)).
	\end{align*}
	and it is straightforward to see that $G$ is self-adjoint and commutes with $d$ and $d^*$.
	As in Ballmann's book, one shows that $\gamma=-G\partial^*\bar{\partial}^*G\alpha$ does the job in both cases. The estimate on $\gamma$ follows from construction. For $\alpha\in H^k_{\delta}(\Lambda^{p,q}M)$ with $\delta\leq -n+1$, we get $\gamma\in H^{k+2}_{\delta'+2}(\Lambda^{p,q}M)$ for any $\delta'> -n+1$. But as in \cite[Theorem 8.4.4]{Joy-Hol-Boo}, the fact that $\alpha=\partial\bar{\partial}\gamma$ allows us to deduce $\gamma\in H^{k+2}_{\delta+2}(\Lambda^{p,q}M)$.
\end{proof}
Let $(M,g,J)$ be a K\"{a}hler manifold. An infinitesimal complex deformation is an endomorphism $I:TM\to TM$ that anticommutes with $J$ and satisfies $\bar{\partial}I=0$ and $\bar{\partial}^*I=0$. By the relation $IJ+JI=0$, $I$ can be viewed as a section of $\Lambda^{0,1}M\otimes T^{1,0}M$.

\begin{theo} Let $(M^n,g,J)$ be an ALE K\"{a}hler manifold with a holomorphic volume form, $k>n/2+1$, $\delta\leq-n/2$ nonexceptional and $I\in H^k_{\delta}(\Lambda^{0,1}M\otimes T^{1,0}M)$ such that $\bar{\partial}I=0$ and $\bar{\partial}^*I=0$. Then there exists a smooth family of complex structures $J(t)$ with $J(0)=J$ such that $J(t)-J\in H^k_{\delta}(T^*M\otimes TM)$ and $J'(0)=I$.
\end{theo}
\begin{proof}
	The proof follows along the lines of Tian's proof by the power series approach \cite{Tian-Smoothness}:
	%(which was adapted in many situations, see [Hajo]):  
	We write $J(t)=J(1-I(t))(1+I(t))^{-1}$, where $I(t)\in  H^k_{\delta}(\Lambda^{0,1}M\otimes T^{1,0}M)$
	and $I(t)$ has to solve the equation
	\begin{align*}
	\bar{\partial}I(t)+\frac{1}{2}[I(t),I(t)]=0,
	\end{align*}
	where $[.,.]$ denotes the Fr\"{o}licher-Nijenhuis bracket. If we write $I(t)$ as a formal power series
	$I(t)=\sum_{k\geq 1} I_kt^k$, the coefficients have to solve the equation
	\begin{align*}
	\bar{\partial} I_N+\frac{1}{2}\sum_{k=1}^{N-1}[I_k,I_{N-k}]=0,
	\end{align*}
	inductively for all $N\geq 2$. 
	As $\Lambda^{n,0}M$ is trivial,
	there is a natural identification of the bundles $\Lambda^{0,1}M\otimes T^{1,0}M=\Lambda^{n-1,1}M$ by using the holomorphic volume form and we now think of the $I_k$ as being $(n-1,1)$-forms.
	Initially, we have chosen
	$I_1\in H^k_{\delta}(\Lambda^{0,1}M\otimes T^{1,0}M)$, given by $I=2I_1J$. 
	By the multiplication property of weighted Sobolev spaces  \cite[p. 538]{Cho-Bru-Boo}, $[I_1,I_1]\in H^{k-1}_{\delta-1}(\Lambda^{n-1,2}M)$. Using $\partial I_1=0$ and $\bar{\partial}^*I_1=0$, one can now show that $\bar{\partial}[I_1,I_1]=0$ and $[I_1,I_1]$ is $\partial$-exact. The $\partial\bar{\partial}$-lemma now implies the existence of a $\psi\in  H^{k+1}_{\delta+1}(\Lambda^{n-2,1}M)$ such that $$\partial\bar{\partial}\psi=-\frac{1}{2}[I_1,I_1],$$ and so, $I_2=\partial \psi\in  H^{k}_{\delta}(\Lambda^{n-1,1}M)$ does the job.
	Inductively, we get a solution of the equation
	\begin{align*}
	\partial\bar{\partial}\psi=\frac{1}{2}\sum_{k=1}^{N-1}[I_k,I_{N-k}],
	\end{align*}
	by the $\partial\bar{\partial}$-lemma since the right hand side is $\bar{\partial}$-closed and $\partial$-exact (which in turn is true because $\partial I_k=0$ for $1\leq k\leq N-1$).
	Now we can choose $I_N=\partial\psi\in   H^{k}_{\delta}(\Lambda^{n-1,1}M)$.
	
	%		 Now if all $I_1,\ldots I_{N-1}\in H^k_{\delta}$,
	%		then $\frac{1}{2}\sum_{k=1}^{N-1}[I_k,I_{N-k}]\in H^{k-1}_{\delta+1}$ by the multiplication property
	%		of weighted Sobolev spaces, see e.g.\ Choquet-Bruhat, p.538. (The assumptions needed for that are  $k>n/2+1$ and $\delta>-n/2$.) Thus, $\psi\in H^{k+1}_{\delta-1}$ and therefore $I_N=\partial\psi\in H^{k}_{\delta}$. 
	Let us prove the convergence of the above series: Let $D_1$ be the constant in the estimate of the $\partial\bar{\partial}$-lemma and $D_2$ be the constant such that
	\begin{align*}
	\left\|[\phi,\psi]\right\|_{H^{k-1}_{\delta-1}}\leq D_2\left\|\phi\right\|_{H^k_{\delta}}\left\|\psi\right\|_{H^k_{\delta}}.
	\end{align*}
	Then one can easily show by induction that 
	\begin{align*}
	\left\|I_N\right\|_{H^k_{\delta}}\leq C(N)\cdot[\frac{1}{2}D_1\cdot D_2]^{N-1}(\left\|I_1\right\|_{H^{k}_{\delta}})^N
	\end{align*}
	for $N\geq1$,
	where $C(N)$ is the sequence defined by $C(1)=1$ and $C(N)=\sum_{i=1}^{N-1}C(i)\cdot C(N-i)$ for $N>1$. By defining $D:=2/(D_1\cdot D_2)$ and $s=\frac{1}{2}D_1\cdot D_2\cdot\left\|I_1\right\|_{H^{k}_{\delta}}\cdot t$, we get
	\begin{align*}
	\left\|I(t)\right\|_{H^k_{\delta}}\leq \sum_{i=1}^{\infty}\left\|I_i\right\|_{H^k_{\delta}} t^i\leq D\cdot\sum_{i=1}^{\infty}C(i)\cdot s^i=D\cdot \left (\frac{1}{2}-\sqrt{\frac{1}{4}-s}\right),
	\end{align*} 
	if $s<1/4$ which shows that the series converges. Thus $I(t)\in H^{k}_{\delta}(\Lambda^{n-1,1}M)$ and $J(t)-J=-2JI(t)(1+I(t))^{-1}\in H^{k}_{\delta}(\Lambda^{n-1,1}M)\cong H^{k}_{\delta}(\Lambda^{0,1}M\otimes T^{1,0}M)$.
\end{proof}
The proof of the above theorem provides an analytic immersion $\Theta:H^{k}_{\delta}(\Lambda^{0,1}M\otimes T^{1,0}M)\cap \ker_{L^2}(\Delta)\supset U\to H^{k}_{\delta}(T^*M\otimes TM)$ whose image is a smooth manifold of complex structures which we denote by $\mathcal{J}^k_{\delta}$ and whose tangent map at $J$ is just the injection.
\begin{prop}\label{kahler-def}Let $(M,g_0,J_0)$ be an ALE Calabi-Yau manifold, $\delta<2-n$ nonexceptional and
	$\mathcal{J}^k_{\delta}$ be as above. Then there exists a $H^{k}_{\delta}$-neighbourhood $\mathcal{U}$ of $J$ and a smooth map $\Phi:\mathcal{J}^k_{\delta}\cap \mathcal{U}\to\mathcal{M}^k_{\delta}$ which associates to each $J\in \mathcal{J}^k_{\delta}\cap \mathcal{U}$ sufficiently close to $J_0$ a metric $g(J)$ which is $H^k_{\delta}$-close to $g_0$ and K\"{a}hler with respect to $J$. Moreover, we can choose the map $\Phi$ such that $$d\Phi_{J_0}(I)(X,Y)=\frac{1}{2}(g_0(IX,J_0Y)+g_0(J_0X,IY)).$$
\end{prop}
\begin{proof}
	We adapt the strategy of Kodaira and Spencer \cite[Section 6]{Kod-Spe-III}. Let $J_t$ be a family in $\mathcal{J}^k_{\delta}$ and define $J_t$-hermitian forms $\omega_t$ by
	$\Pi^{1,1}_{t}\omega_0(X,Y)=\frac{1}{2}(\omega_0(X,Y)+\omega_0(J_tX,J_tY))$. Let $\partial_t,\bar{\partial}_t$ the associated Dolbeaut operators and $\partial_t^*,\bar{\partial}_t^*$ their formal adjoints with respect to the metric $g_t(X,Y):=\omega_t(X,J_tY)$. We now define a forth-order linear differential operator $E_t:H^k_{\delta}(\Lambda^{p,q}_tM)\to H^{k-4}_{\delta-4}(\Lambda^{p,q}_tM)$ by
	\begin{align*}
	E_t=\partial_t\bar{\partial}_t\bar{\partial}^*_t\partial_t^*+\bar{\partial}^*_t\partial_t^*\partial_t\bar{\partial}_t+\bar{\partial}^*_t\partial_t\partial_t^*\bar{\partial}_t+\partial_t^*\bar{\partial}_t\bar{\partial}^*_t\partial_t
	+\bar{\partial}^*_t\bar{\partial}_t+\partial_t^*\partial_t.
	\end{align*}
	It is straightforward to see that $E_t$ is formally self-adjoint and strongly elliptic. Moreover,	$\alpha\in\ker_{H^k_{\delta}}(E_t)$ if and only if 
	$\partial_t\alpha=0$, $\bar{\partial}_t\alpha=0$ and $\bar{\partial}^*_t\partial_t^*\alpha=0$, i.e.\ $d\alpha=0$ and $\bar{\partial}^*_t\partial_t^*\alpha=0$ hold simultaneously. If $\delta$ is nonexceptional, $E_t$ is Fredholm which allows to define for each $t$ its Greens operator $G_t:H^{k-4}_{\delta-4}(\Lambda^{p,q}_tM)\to H^k_{\delta}(\Lambda^{p,q}_tM)$.
	As in \cite[Proposition 7]{Kod-Spe-III}, one now shows that
	\begin{align*}
	\ker_{L^2}(d)\cap H^k_{\delta}(\Lambda^{p,q}_tM)=\partial_t\bar{\partial}_t(H^{k+2}_{\delta+2}(\Lambda^{p-1,q-1}_tM))\oplus \ker_{L^2}(E_t)\cap H^k_{\delta}(\Lambda^{p,q}_tM),
	\end{align*}
	is an $L^2(g_t)$ orthogonal decomposition.
	The dimension of $\ker_{L^2}(E_t)\cap H^k_{\delta}(\Lambda^{1,1}_tM)$ is constant for small $t$ which implies that $G_t$ depends smoothly on $t$. The proof of this fact is exactly as in \cite[Proposition 8]{Kod-Spe-III}.
	
	Now observe that
	$E_t\omega_t\in H^{k-4}_{\delta-4}(\Lambda^{1,1}_tM)$ if $\omega_t$ and $J_t$ are $H^k_{\delta}$-close to $\omega_0$ and $J_0$, respectively. This allows us to define
	\begin{align*}
	\tilde{\omega}_t&:=\omega_t-G_tE_t\omega_t+\partial_t\bar{\partial}_tu_t
	=(1-G_tE_t)\Pi^{1,1}_t\omega_0+\partial_t\bar{\partial}_tu_t,
	\end{align*}
	where $u_t\in H^{k+2}_{\delta+2}(M)$ is a smooth family of functions such that $u_0=0$ which will be defined later. Clearly, $$\bar{\omega}_t:=\omega_t-G_tE_t\omega_t\in \ker E_t.$$ As $\bar{\omega}_t$ is $H^k_{\delta}$-close to $\omega_0$, $\nabla^{g_t}\bar{\omega}_t\in H^{k-1}_{\delta-1}(g_t)$, since $\omega_0$ is $g_0$-parallel. Therefore, $\nabla^{g_t}\bar{\omega}_t=\textit{O}(r^{-\alpha-1}) $ and $\nabla^{g_t,2}\bar{\omega_t}=\textit{O}(r^{-\alpha-2})$ for any $\alpha<-\delta$. Thus, if we choose the nonexceptional value $\delta$ so that $\delta<-n+2$, integration by parts implies that 
	\begin{align*}
	%\left\|\bar{\partial}^*_t\partial_t^* \tilde{\omega}_t\right\|_{L^2(g_t)}+		
	%\left\|\partial_t\bar{\partial}_t\tilde{\omega}_t \right\|_{L^2(g_t)}+
	%\left\|\partial_t^*\bar{\partial}_t\tilde{\omega}_t \right\|_{L^2(g_t)}+
	%\left\|\bar{\partial}^*_t\partial_t \tilde{\omega}_t\right\|_{L^2(g_t)}+
	\left\|\partial_t \bar{\omega}_t\right\|^2_{L^2(g_t)}+
	\left\|\bar{\partial_t}\bar{\omega}_t \right\|^2_{L^2(g_t)}
	\leq (E_t\bar{\omega}_t,\bar{\omega}_t)_{L^2(g_t)}=0.
	\end{align*}
	Therefore, $\bar{\omega}_t$ and hence also $\tilde{\omega}_t$ is closed. 
	Differentiating at $t=0$ yields
	\begin{align*}
	\tilde{\omega}'_0=(1-G_0E_0)\omega'_0-G_0(E_0'\omega_0)+\partial_0\bar{\partial}_0u'_0=\omega'_0-G_0(E_0'\omega_0)+\partial_0\bar{\partial}_0u'_0
	\end{align*}
	Because $d\tilde{\omega}_t=0$, we have $d\tilde{\omega}'_0=0$ and since $J_0'$ is an infinitesimal complex deformation, $E_0\omega'_0=0$ and $d\omega'_0=0$ 
	%		(\textcolor{red}{check!})
	which implies that  		 
	$$G_0(E_0'\omega_0)\in \ker_{L^2}(E_0)^{\perp}\cap\ker_{L^2}(d)\cap H^{k}_{\delta}(\Lambda^{1,1}_0M)=\partial_0\bar{\partial}_0(H^{k+2}_{\delta+2}(M)).$$ Let now $ v\in H^{k+2}_{\delta+2}(M) $ so that $\partial_0\bar{\partial}_0v=G_0(E_0'\omega_0).$
	Then, define $u_t\in H^{k+2}_{\delta+2}(M) $ by $$u_t:=tv.$$
	By this choice, $
	\tilde{\omega}_0'=\omega'_0$
	and the assertion for $d\Phi_{J_0}(J'_0)=\tilde{g}_0'$ follows immediately.
	Finally, $\tilde{g}_t(X,Y):=\tilde{\omega}_t(X,J_tY)$ is a Riemannian metric for $t$ small enough and it is K\"{a}hler with respect to $J_t$.

\end{proof}
\begin{rk}
	Let $J_t$ is a smooth family of complex structures in $\mathcal{J}^k_{\delta}\cap \mathcal{U}$ and $g_t=\Phi(J_t)$. Then the construction in the proof above shows that $I=J'_0$ and $h=g'_0$ are related by
	\begin{align*}
	h(JX,Y)=-\frac{1}{2}(g(X,IY)+g(IX,Y)).
	\end{align*}
\end{rk}
\noindent Before we state the next theorem, recall the notation $\mathcal{G}_{\delta}^k$ we used in Proposition \ref{slice}.
\begin{theo}\label{CY-ALE-Int}
	Let $(M^n,g_0,J_0)$ be an ALE Calabi-Yau manifold and $\delta\in (1-n,2-n)$ nonexceptional. Then for any $h\in\ker_{L^2}(L_{g_0})$, there exists a smooth family  $g(t)$ of Ricci-flat metrics in $\mathcal{G}^k_{\delta}$ with $g(0)=g_0$ and $g'_0=h$. Each metric $g(t)$ is ALE and K\"{a}hler with respect to some complex structure $J(t)$ which is $H^k_{\delta}$-close to $J_0$. In particular, $g_0$ is integrable.
\end{theo}
\begin{proof}
	We proceed similarly as in \cite[Chapter 12]{Besse}, except the fact that we use weighted Sobolev spaces. Given a complex structure $J$ close to $J_0$ and a $J$-$(1,1)$-form $\omega$ which is $H^{k}_{\delta}$-close to $\omega_0$, we seek a Ricci-flat metric in the cohomology class $[\omega]\in \mathcal{H}^{1,1}_J(M)$. As the first Chern class vanishes, there exists a function $f_{\omega}\in H^{k}_{\delta}(M)$, such that $i\partial\bar{\partial}f_{\omega}$ is the Ricci form of $\omega$. If $\bar{\omega}\in [\omega]$ and $\bar{\omega}-\omega\in H^{k}_{\delta}(\Lambda^{1,1}_{J}M)$, the $\partial\bar{\partial}$-lemma implies that there is a $u\in H^{k+2}_{\delta+2}(M)$ such that $\bar{\omega}=\omega+i\cdot \partial\bar{\partial}u$. Ricci-flatness of $\bar{\omega}$ is now equivalent to the condition
	\begin{align*}
	f_{\omega}=\log\frac{(\omega+i\partial\bar{\partial}u)^n}{\omega^n}=:Cal(\omega,u).
	\end{align*}
	Let $\mathcal{J}^k_{\delta}$ be as above and $\Delta_J$ the Dolbeaut Laplacian of $J$ and the metric $g(J)$. Then all the $(L^2_{\delta})$-cohomologies $\mathcal{H}^{1,1}_{J,\delta}(M)=\ker_{L^2_{\delta}}(\Delta_J)\cap L^2_{\delta}(\Lambda^{1,1}M)$ are isomorphic for $J\in\mathcal{J}^k_{\delta}$ if we  $\mathcal{J}^k_{\delta}$ is small enough:
	We have $\mathcal{H}^2_{\delta}(M)=\mathcal{H}^{2,0}_{J,\delta}(M)\oplus \mathcal{H}^{1,1}_{J,\delta}(M)\oplus \mathcal{H}^{0,2}_{J,\delta}(M)$. The left hand side is independent of $J$ and the metric $g(J)$ is provided by Proposition \ref{kahler-def}. The spaces on the right hand side are kernels of $J$-dependent elliptic operators
	whose dimension depends upper-semicontinuously on $J$. However the sum of the dimensions is constant
	and so the dimensions must be constant as well.
	
	Thus, there is a natural projection $pr_{J}:\ker_{L^2}(\Delta_{J_0})\to\ker_{L^2}(\Delta_{J})$ which is an isomorphism. We now want to apply the implicit function theorem to the map
	\begin{align*}
	G: \mathcal{J}^k_{\delta}\times \mathcal{H}^{1,1}_{J_0,\delta}(M)\times H^{k+2}_{\delta+2}(M)&\to H^{k}_{\delta}(M)\\
	(J,\kappa,u)&\mapsto Cal(\omega(J)+pr_J(\kappa),u)-f_{\omega(J)+pr_J(\kappa)},
	\end{align*}
	where $\omega(J)(X,Y):=g(J)(JX,Y)$ and $g(J)$ is the metric constructed in Proposition \ref{kahler-def}.
	We have $G(J_0,0,0)=0$ and the differential restricted to the third component is just given by $\Delta:H^{k+2}_{\delta+2}(M)\to  H^{k}_{\delta}(M)$ (c.f. \cite[p.\ 328]{Besse}), which is an isomorphism. Therefore we find a map $\Psi$ such that $G(J,\kappa, \Psi(J,\kappa))=0$.
	
	Let now $h\in\ker_{L^2}(L_{g_0})$ and let $h=h_H+h_A$ its decomposition into a $J_0$-hermitian and a $J_0$-antihermitian part. We want to show that $h$ is tangent to a family of Ricci-flat metrics.
	We have seen in Theorem \ref{ell-reg-prop} together with Remark \ref{lin-decay} that $h\in H^k_{\delta}(S^2T^*M)$ for all $\delta>1-n$ and we can define $I\in H^k_{\delta}(T^*M\otimes TM)$ and $\kappa\in H^k_{\delta}(\Lambda^{1,1}_{J_0})(M)$ by
	\begin{align}\label{herm+antiherm}
	g(X,IY)=-h_A(X,J_0Y),\qquad \kappa(X,Y)=h_H(J_0X,Y).
	\end{align}
	It is easily seen that $I$ is a symmetric endomorphism satisfying $IJ_0+J_0I=0$ and thus can be viewed as 
	$I\in H^k_{\delta}(\Lambda^{0,1}M\otimes T^{1,0}M)$. Moreover, because $h_A$ is a $TT$-tensor, $\bar{\partial} I=0$ and $\bar{\partial}^*I=0$. In addition $\kappa\in \mathcal{H}^{1,1}_{J_0}(M)$. The proof of this facts is as in \cite{Koiso-Cx-Str}.
	Let $J(t)=\Theta(t\cdot I)$ be a family of complex structures tangent to $I$ and $\tilde{\omega}(t)=\tilde{\Phi}(J(t))$ be the associated family of K\"{a}hler forms.
	We consider the family $\omega(t)=\tilde{\omega}(t)+pr_{J(t)}(t\cdot \kappa)+i\partial\bar{\partial} \Psi(J(t),t\cdot\kappa)$ and the associated family of Ricci-flat metrics $\tilde{g}(t)(X,Y)=\omega(t)(X,J(t)Y)$. It is straightforward that $\tilde{g}'(0)=h$. By Proposition \ref{slice}, 
	there exist diffeomorphisms $\varphi_t$ with $\varphi_0=id$ such that $g(t)=\varphi_t^*\tilde{g}(t)\in \mathcal{G}^k_{\delta}$. We obtain $g'(0)=h+\Li_{X}g_0$ for some $X\in H^{k+1}_{\delta+1}(TM)$.
	Since $h$ is a TT-tensor due to Lemma \ref{tttensors}, $h\in T_{g_0}\mathcal{G}_{\delta}^{k}$. On the other hand, $g'(0)\in  T_{g_0}\mathcal{G}_{\delta}^{k}$ as well
	which implies that $g'(0)=h$ due to the decomposition in Proposition \ref{slice}.
	
	By Theorem \ref{analyticset}, the set of stationary solutions of the Ricci-DeTurck flow $\mathcal{F}$ close to $g_0$ is an analytic set contained in a finite-dimensional manifold $\mathcal{Z}$ with $T_{g_0}\mathcal{Z}=\ker_{L^2}(L_{g_0})$.
	The above construction provides a smooth map $\Xi:\ker_{L^2}(L_{g_0})\supset \mathcal{U}\to\mathcal{F}\subset\mathcal{Z}$ whose tangent map is the identity. Therefore, there exists a $L^2\cap L^{\infty}$-neighbourhood $\mathcal{U}$ of $g_0$ in the space of metrics such that $\mathcal{F}\cap\mathcal{U}=\mathcal{Z}\cap\mathcal{U}$.
\end{proof}
Let $h\in C^{\infty}(S^2T^*M)$ and $h_H,h_A$ its hermitian and anti-hermitian part, respectively.
The hermitian and anti-hermitian part are preserved by $L_{g_0}$. Let $I=I(h_A)$ and $\kappa=\kappa(h_H)$ be defined as in \eqref{herm+antiherm}. Then we have the relations
$I(L(h_A))=\Delta_C(I(h_A))$ and $\kappa(L(h_H) )=\Delta_H(\kappa(h_H))$, where $\Delta_C$ and $ \Delta_H$ are the complex Laplacian and the Hodge Laplacian acting on $C^{\infty}(\Lambda^{0,1}M\otimes T^{1,0}M)$ and $C^{\infty}(\Lambda^{1,1}_{J_0}M)$, respectively. 
For details see \cite{Koiso-Cx-Str} and \cite[Chap. $12$]{Besse}.
As a consequence, we get
\begin{theo}[Koiso]
	If $(M,g_0,J_0)$ is an ALE Ricci-flat K\"{a}hler manifold, it is linearly stable.
\end{theo}

\section{Ricci flow}\label{RF-Sec}
Our main result of this section is the following
\begin{theo}\label{main-theo}
	Let $(M^n,g_0)$ be an ALE Ricci-flat manifold. Assume it is linearly stable and integrable. Then for every $\epsilon>0$, there exists a $\delta>0$ such that the following holds: For any metric $g\in \mathcal{B}_{L^2\cap L^{\infty}}(g_0,\delta)$, there is a complete Ricci-DeTurck flow $(M^n,g(t))_{t\geq 0}$ starting from $g$ converging to an ALE Ricci-flat  metric $g_{\infty}\in \mathcal{B}_{L^2\cap L^{\infty}}(g_0,\epsilon)$. 
\end{theo}

\subsection{An expansion of the Ricci flow}\label{sec-equ-flo}

Let us fix an ALE Ricci-flat manifold $(M^n,g_0)$ once and for all.
Recall the definition of the Ricci flow
\[
\left\{
\begin{array}{rl}
&\partial_tg=-2\Ric(g(t)), \quad\mbox{on}\quad M\times (0,+\infty),\\
&\\
& g(0)=g_0+h,
\end{array}
\right.
\]
where $h$ is a symmetric $2$-tensor on $M$ (denoted by $h\in S^2T^*M$) such that $g(0)$ is a metric.
The Ricci-DeTurck flow is given by
\[
\left\{
\begin{array}{rl}
&\partial_tg=-2\Ric(g(t))+\Li_{V(g(t),g_0)}(g(t)), \quad\mbox{on}\quad M\times (0,+\infty),\\
&\\
& g(0)=g_0+h,
\end{array}
\right.
\]
where $V(g(t),g_0)$ is a vector field defined locally by $V^k=g(t)^{ij}(\Gamma(g(t))_{ij}^k-\Gamma(g_0)_{ij}^k)$ and globally by
\begin{eqnarray}\label{de-turck-vect}
g_0(V(g(t),g_0),.):=-\div_{g(t)}g_0+\frac{1}{2}\nabla^{g(t)}\tr_{g(t)}g_0.
\end{eqnarray} 
Following \cite[Lemma 2.1]{Shi-Def}, the Ricci-DeTurck flow can be written in coordinates as
\begin{eqnarray*}
	\partial_tg_{ij}&=&g^{ab}\nabla^{g_0,2}_{ab}g_{ij}-g^{kl}g_{ip}\Rm(g_0)_{jklp}-g^{kl}g_{jp}\Rm(g_0)_{iklp}\\
	&&+g^{ab}g^{pq}\left(\frac{1}{2}\nabla^{g_0}_ig_{pa}\nabla^{g_0}_jg_{qb}+\nabla^{g_0}_ag_{jp}\nabla^{g_0}_qg_{ib}\right)\\
	&&-g^{ab}g^{pq}\left(\nabla^{g_0}_ag_{jp}\nabla^{g_0}_bg_{iq}-\nabla^{g_0}_jg_{pa}\nabla^{g_0}_bg_{iq}-\nabla^{g_0}_ig_{pa}\nabla^{g_0}_bg_{jq}\right).
\end{eqnarray*}
For our purposes, we calculate a different expansion:
Let $\bar{g}$ and $g$ two Riemannian metrics on a given manifold and $h:=g-\bar{g}$. Then a careful computation shows that in local coordinates,
\begin{align*}
2(\mathrm{Ric}(g)_{ij}-\mathrm{Ric}(\bar{g})_{ij})&=-(L_{\bar{g}} h)_{ij}+\bar{g}^{uv}(\nabla^{\bar{g},2}_{iu}h_{jv}+\nabla^{\bar{g},2}_{ju}h_{iv}-\nabla^{\bar{g},2}_{ij}h_{uv})\\
&\quad +(g^{uv}-\bar{g}^{uv})(\nabla^{\bar{g},2}_{ui}h_{jv}+\nabla^{\bar{g},2}_{uj}h_{iv}-\nabla^{\bar{g},2}_{uv}h_{ij}-\nabla^{\bar{g},2}_{ij}h_{uv})\\
&\quad+g^{uv}g^{pq}(\nabla^{\bar{g}}_uh_{pi}\nabla^{\bar{g}}_vh_{qj}-\nabla^{\bar{g}}_ph_{ui}\nabla^{\bar{g}}_vh_{qj}+\frac{1}{2}\nabla^{\bar{g}}_ih_{up}\nabla^{\bar{g}}_jh_{vq})\\
&\quad+g^{uv}(-\nabla^{\bar{g}}_uh_{vp}+\frac{1}{2}\nabla^{\bar{g}}_ph_{uv})g^{pq}(\nabla^{\bar{g}}_ih_{qj}+\nabla^{\bar{g}}_jh_{qi}-\nabla^{\bar{g}}_qh_{ij}),
\end{align*}
where $g^{uv},\bar{g}^{uv}$ are the inverse matrices of $g_{uv},\bar{g}_{uv}$, respectively.
For a calculation, see for instance \cite[p. 15]{Bam-Phd}. 
Furthermore, if a background metric $g_0$ is fixed and if  $V=V(g,g_0)$ is defined as above, then we have the expansion
\begin{align*}
V(g&,g_0)^k-{V}(\bar{g},g_0)^k=\frac{1}{2}\bar{g}^{ij}\bar{g}^{kl}(\nabla^{\bar{g}}_ih_{jl}+\nabla^{\bar{g}}_jh_{il}-\nabla^{\bar{g}}_lh_{ij})
-h_{pq}\bar{g}^{pi}\bar{g}^{qj}(\Gamma(\bar{g})_{ij}^k-\Gamma(g_0)_{ij}^k)
\\
&\quad+\frac{1}{2}\bar{g}^{ij}(g^{kl}-\bar{g}^{kl})(\nabla^{\bar{g}}_ih_{jl}+\nabla^{\bar{g}}_jh_{il}-\nabla^{\bar{g}}_lh_{ij})+\frac{1}{2}\bar{g}^{kl}(g^{ij}-\bar{g}^{ij})(\nabla^{\bar{g}}_ih_{jl}+\nabla^{\bar{g}}_jh_{il}-\nabla^{\bar{g}}_lh_{ij})
\\&\quad+\frac{1}{2}(g^{ij}-\bar{g}^{ij})(g^{kl}-\bar{g}^{kl})(\nabla^{\bar{g}}_ih_{jl}+\nabla^{\bar{g}}_jh_{il}-\nabla^{\bar{g}}_lh_{ij})-h_{pq}(g^{pi}-\bar{g}^{pi})\bar{g}^{qj}(\Gamma(\bar{g})_{ij}^k-\Gamma(g_0)_{ij}^k).
\end{align*}
Thus for $V=V(g,g_0)$ and $\bar{V}=V(\bar{g},g_0)$, we have
\begin{align*}
\Li_Vg_{ij}-\Li_{\bar{V}}\bar{g}_{ij}&=\Li_{V}\bar{g}_{ij}+\Li_{V}h_{ij}-\Li_{\bar{V}}\bar{g}_{ij}\\ &=\nabla^{\bar{g}}_iV_j+\nabla^{\bar{g}}_jV_i+V^k\nabla^{\bar{g}}_kh_{ij}+\nabla^{\bar{g}}_iV^kh_{kj}+\nabla^{\bar{g}}_jV^kh_{ik}-\Li_{\bar{V}}\bar{g}_{ij}.
\end{align*}
Now if $\bar{g}$ is a Ricci-flat metric that additionally satisfies $\bar{V}=0$, we can write the Ricci-DeTurck flow as an evolution of the difference $h(t):=g(t)-\bar{g}$ for which we get
\begin{equation}\begin{split}\label{rdt-expansion}
\partial_t h=\partial_tg&=-2\Ric(g)+2\mathrm{Ric}_{\bar{g}}+\Li_{V(g,g_0)}g-\Li_{V(\bar{g},g_0)}\bar{g}\\
&=L_{\bar{g}}h-\Li_{\langle h,\Gamma(\bar{g})-\Gamma(g_0)\rangle}\bar{g}+F*\nabla^{\bar{g}} h*\nabla^{\bar{g}} h+\nabla^{\bar{g}}(G*h*\nabla^{\bar{g}} h),
\end{split}
\end{equation}
where $\langle h,\Gamma(\bar{g})-\Gamma(g_0)\rangle^k=h_{pq}\bar{g}^{pi}\bar{g}^{qj}(\Gamma(\bar{g})_{ij}^k-\Gamma(g_0)_{ij}^k)$ and $*$ denotes a linear combination of tensor products and contractions with respect to the metric $\bar{g}$. The tensors $F$ and $G$ depend on $g^{-1}$ and $\Gamma(g_0)$.
%The tensors $F(g,\bar{g},g_0)$ and $G(h)$ are bounded  and satsifty $|\nabla^k F|\leq C|\nabla^k h|$ and  $|\nabla^k G|\leq C|\nabla^k h|$ if $\left\|h\right\|_{L^{\infty}}<\epsilon$ for some small enough $\epsilon>0$.

\subsection{Short-time estimates and an extension criterion}\label{Short-time}
In this subsection we recall the short-time estimates of $C^k$-norms and an extension criterion for the Ricci-DeTurck flow. In addition, we prove some new Shi-type estimates for $L^2$-type Sobolev norms.
For the sake of simplicity, all covariant derivatives and norms in this subsection are taken with respect to $g_0$.
\begin{lemma}[A priori short-time $C^k$-estimates]\label{Ck-estimate}Let $(M,g_0)$ be a complete Ricci-flat manifold of bounded curvature. Then there exist constants $\epsilon>0$ and $\tau\geq 1$ such that if $g(0)$ is a metric satisfying
	\[\left\|g(0)-g_0\right\|_{L^{\infty}}<\epsilon,
	\]
	there exists a Ricci-DeTurck flow $(g(t))_{t\in[0,\tau]}$ with initial metric $g(0)$ which satisfies the estimates
	\[\left\|\nabla^k( g(t)-g_0)\right\|_{L^{\infty}}<C(k,\tau)t^{-k/2}\left\|g(0)-g_0\right\|_{L^{\infty}},\quad \forall  k\in \mathbb{N}_0,\quad t\in (0,\tau].
	\]
	Moreover, $(g(t))_{t\in[0,\tau)}$ is the unique Ricci-DeTurck flow starting at $g(0)$ which satisfies
	\[\left\| g(t)-g_0\right\|_{L^{\infty}}<C(0,\tau)\left\|g(0)-g_0\right\|_{L^{\infty}}. \]
	In particular, this implies the following: if $(g(t))_{t\in[0,\infty)}$ is a Ricci-DeTurck flow and such that it is in $\mathcal{B}_{L^{\infty}}(g_0,\epsilon)$ for all time, then there exist constants such that
	\[\left\|\nabla^k( g(t)-g_0)\right\|_{L^{\infty}}<C(k)\epsilon, \qquad \forall  k\in \mathbb{N},\quad t\in [1,\infty).
	\] 
\end{lemma}
\begin{proof}
	The same statement is given in \cite[Proposition 2.8]{Bam} for the case of negative Einstein metrics. The proof is standard and translates easily to the present situation.
	For more details, see e.g. \cite[Section 3.7]{Bam-Phd}.
\end{proof}
\begin{lemma}[A priori short-time $L^2$-estimate]\label{lemma-short-time-L2}
	Let $(M,g_0)$ be an ALE Ricci-flat manifold. Then there exists an $\epsilon=\epsilon(n,g_0)>0$ with the following property: Suppose that $(g(t))_{t\in[0,T_{max})}$ is a Ricci-DeTurck flow such that $h(t)=g(t)-g_0$ satisfies $\|h(t)\|_{L^{\infty}}<\epsilon$ for all $t\in [0,T_{max})$ and $\|h(0)\|_{L^2}<\infty$. Then, $\|h(t)\|_{L^2}<\infty$ for all $t\in (0,T_{max})$ and there exists a constant $C=C(n,g_0)$ such that 
	\begin{align*}
	\|h(t)\|_{L^2}\leq e^{Ct}\cdot\|h(0)\|_{L^2},\qquad \forall t\in (0,T_{max}).
	\end{align*}
\end{lemma}	
\begin{proof}
	By \eqref{rdt-expansion}, we can rewrite the Ricci-DeTurck flow with gauge $g_0$ in the schematic form
	\begin{equation}\begin{split}\label{rdt-expansion2}
	\partial_t h=\Delta h+Rm*h+F*\nabla h*\nabla h+\nabla(G*h*\nabla h).
	\end{split}
	\end{equation}
	For each $R>0$, let $\eta_R:[0,\infty)$ be a function such that $\eta_R(r)=1$ for $r\leq R$, $\eta_R(r)=0$ for $r\geq 2R$ and
	$|\nabla\eta_R|\leq 2/R$. For $x\in M$, let $\phi_{R,x}(y)=\eta_R( d(x,y))$. Then $\phi_{R,x}\equiv 1$ on $B_R(x)$, $\phi_{R,x}\equiv 0$ on $M\setminus B_{2R}(x)$ and $|\nabla\phi_{R,x}|\leq 2/R$. For notational convenience, we write $\phi=\phi_{R,x}$ in the following.
	By \eqref{rdt-expansion2}, we obtain
	\begin{align*}
	\partial_t \int_M |h|^2\phi^2 d\mu&\leq 2\int_M\langle \Delta h,\phi^2h\rangle d\mu+ C\|Rm\|_{L^{\infty}}\int_M |h|^2\phi^2d\mu\\&+C\|h\|_{L^{\infty}}\int_M |\nabla h|^2\phi^2 d\mu+\int_M \langle \nabla(G*h*\nabla h), h\rangle \phi^2 d\mu\\
	&\leq -2\int_M |\nabla h|^2\phi^2 d\mu + C\int_M |\nabla h||h||\nabla\phi|\phi d\mu\\
	& +C(g_0)\int_M |h|^2\phi^2d\mu+C\|h\|_{L^{\infty}}\int_M |\nabla h|^2\phi^2 d\mu\\
	&\leq (-2+C\epsilon+C\delta  )\int_M |\nabla h|^2\phi^2 d\mu+C(g_0)\int_M |h|^2\phi^2d\mu+\frac{C}{\delta}\int_M |h|^2|\nabla \phi|^2d\mu\\
	&\leq (C(g_0)+\frac{2C}{\delta R^2})\int_{B_{2R}(x)} |h|^2d\mu 
	\end{align*}
	for an appropriate choice of $\delta$. Define
	\begin{align*}
	A(t,R)=\sup_{x\in M}\int_M |h(t)|^2\phi_{R,x}^2d\mu.
	\end{align*}
	As $(M,g_0)$ is ALE, there exists a constant $N=N(n)$ such that each ball on $M$ of radius $2R$ can be covered by $N$ balls of radius $R$. Thus, by integration in time,
	\begin{align*}
	\int_M |h(t)|^2\phi_{R,x}^2 d\mu&\leq  \int_M |h(0)|^2\phi_{R,x}^2 d\mu+(C(g_0)+\frac{2C}{\delta R^2})\int_0^t\int_{B_{2R}(x)} |h(s)|^2d\mu ds \\
	&\leq  \int_M |h(0)|^2\phi_{R,x}^2 d\mu+N(C(g_0)+\frac{2C}{\delta R^2})\int_0^t A(s,R)ds.
	\end{align*}
	Consequently,
	\begin{align*}
	A(t,R)\leq A(0,R)+ N(C(g_0)+\frac{2C}{\delta R^2})\int_0^t A(s,R)ds,
	\end{align*}
	and by the Gronwall inequality,
	\begin{align*}
	A(t,R)\leq A(0,R)\cdot \exp\left(N(C(g_0)+\frac{2C}{\delta R^2})t\right).
	\end{align*}
	The assertion follows from letting $R\to\infty$.
\end{proof}

\begin{lemma}[A priori short-time $H^k$-estimates]\label{Hk-estimate}
	Let $(M,g_0)$ be an ALE Ricci-flat manifold. Then there exists an $\epsilon=\epsilon(n,g_0)>0$ with the following property: Suppose that $(g(t))_{t\in[0,T_{max})}$ is a Ricci-DeTurck flow such that $h(t)=g(t)-g_0$ satisfies 
	\begin{eqnarray*}
		\|h(t)\|_{L^{\infty}}<\epsilon,\quad \forall t\in [0,T_{max}).
	\end{eqnarray*}
	Then for each $T\in (0,T_{max})$ and $k\in\mathbb{N}$ there exist constants $C_k=C_k(n,g_0,T)$ such that if $\|h(t)\|_{L^2}\leq K$ for all $t\in [0,T]$, we get
	\begin{align*}
	\|\nabla^kh(t)\|_{L^2}\leq C_k\cdot t^{-k/2}\cdot ,K\qquad \forall t\in (0,T].
	\end{align*}
	In particular, if $(g(t))_{t\in[0,T_{max})}$ is a Ricci flow satisfying $\|h(t)\|_{L^{\infty}}<\epsilon$  and $\|h(t)\|_{L^2}<K$ as long as $t\in [0,T_{max})$, then there exist constants $C_k=C_k(n,g_0)$ such that 
	\begin{align*}
	\|\nabla^k h(t)\|_{L^2}\leq C_k\cdot K,\qquad \forall k\in\mathbb{N},\qquad\forall t\in [1,T_{max}).
	\end{align*}
\end{lemma}
\begin{proof}
	The proof follows from a delicate argument involving a sequence of cutoff functions.	
	By differentiating \eqref{rdt-expansion2}, we get
	\begin{align*}
	\partial_t \nabla^kh&=\nabla^k\Delta h+\nabla^k(Rm*h)+\nabla^k(F*\nabla h*\nabla h)+\nabla^{k+1}(G*h*\nabla h)\\
	&=\Delta\nabla^k h+\sum_{l=0}^k\nabla^lRm*\nabla^{k-l}h+\sum_{\substack{0\leq l_1,l_2,l_3\leq k\\l_1+l_2+l_3=k}}\nabla^{l_1}F*\nabla^{l_2+1}h*\nabla^{l_3+1}h\\
	&\quad+\nabla\left(\sum_{\substack{0\leq l_1,l_2,l_3\leq k\\l_1+l_2+l_3=k}}\nabla^{l_1}G*\nabla^{l_2} h*\nabla^{l_3+1}h\right).
	\end{align*}
	%	Observe that without loss of generality, we may start the summation index of the second last term at $1$.
	Let $\phi$ be a cutoff function  as in the proof of Lemma \ref{lemma-short-time-L2}. Then,
	\begin{align*}
	\partial_t\int_M |\nabla^kh|^2\phi^2 d\mu&\leq 2 \int_M\langle \Delta \nabla^k h,\nabla^k h\rangle \phi^2d\mu+C\sum_{l=0}^k \|\nabla^lRm\|_{L^{\infty}}\int_M |\nabla^{k-l}h||\nabla^lh|\phi^2d\mu\\
	&\quad+\sum_{\substack{0\leq l_1,l_2,l_3\leq k\\l_1+l_2+l_3=k}}\int_M\langle\nabla^{l_1}F*\nabla^{l_2+1}h*\nabla^{l_3+1}h,\nabla^kh\rangle\phi^2d\mu\\
	&\quad+\int_M\langle\nabla(\sum_{\substack{0\leq l_1,l_2,l_3\leq k\\l_1+l_2+l_3=k}}\nabla^{l_1}G*\nabla^{l_2} h*\nabla^{l_3+1}h),\nabla^kh\rangle\phi^2d\mu.	
	%	&+C\sum_{l=1}^k\|\nabla^{l+1}h\|_{L^{\infty}}\int_M |\nabla^{k-l+1}h||\nabla^k h|\phi^2d\mu\\
	%	&+\int_M\langle \nabla(\sum_{l=0}^k\nabla^{l+1} h*\nabla^{k-l}h,\nabla^k h\rangle \phi^2 d\mu.
	\end{align*}
	Let us consider each of these terms separately. Then we get 
	\begin{align*}
	2 \int_M\langle \Delta \nabla^k h,\nabla^k h\rangle \phi^2\,d\mu&\leq-2\int_M |\nabla^{k+1}h|^2\phi^2\,d\mu+2\int_M |\nabla^{k+1}h||\nabla^kh||\nabla\phi|\phi \,d\mu\\
	&\leq(-2+\delta)\int_M |\nabla^{k+1}h|^2\phi^2\,d\mu+\frac{1}{\delta}\int_M |\nabla^k h|^2|\nabla\phi|^2\,d\mu
	\end{align*}
	and
	\begin{align*}
	C\sum_{l=0}^k \|\nabla^lRm\|_{L^{\infty}}\int_M |\nabla^{k-l}h||\nabla^kh|\phi^2\,d\mu
	\leq C\sum_{l=0}^k \int_M |\nabla^lh|^2\phi^2\,d\mu.
	\end{align*}
	In the estimates of the higher order terms, we use the property $\|\nabla^k h\|_{L^{\infty}}\leq C_k  t^{-k/2}\epsilon$, which follows from Lemma \ref{Ck-estimate}.
	It also implies
	$\|\nabla^k F\|_{L^{\infty}}\leq C\cdot t^{-k/2}$  and $\|\nabla^k G\|_{L^{\infty}}\leq C\cdot t^{-k/2}$  for $t\in (0,T]$ and $k\in\mathbb{N}$. 
	\begin{align*}
	\sum_{\substack{0\leq l_1,l_2,l_3\leq k\\l_1+l_2+l_3=k}}&\int_M\langle\nabla^{l_1}F*\nabla^{l_2+1}h*\nabla^{l_3+1}h,\nabla^kh\rangle\phi^2\,d\mu
	\\&\leq C \sum_{0\leq l\leq i\leq k}\int_M|\nabla^{k-i}F| |\nabla^{l+1}h||\nabla^{i-l+1}h||\nabla^kh|\phi^2\,d\mu\\
	&\leq C\cdot t\sum_{0\leq l\leq i\leq k}\int_M|\nabla^{k-i}F| |\nabla^{l+1}h||\nabla^{i-l+1}h|\phi^2\,d\mu
	+C t^{-1}\int_M |\nabla^k h|^2\phi^2\,d\mu\\
	&\leq C\epsilon\sum_{l=0}^kt^{-k+l}\int_M |\nabla^{l+1}h|^2\phi^2\,d\mu +C t^{-1}\int_M |\nabla^k h|^2\phi^2\,d\mu \\
	&\leq C\cdot \epsilon \int_M |\nabla^{k+1} h|^2\phi^2d\mu+C\sum_{l=1}^kt^{-k+l-1}\int_M |\nabla^l h|^2\phi^2\,d\mu. 
	\end{align*}
	For the last term, we first perform integration by parts:
	\begin{align*}
	\int_M&\langle\nabla(\sum_{\substack{0\leq l_1,l_2,l_3\leq k\\l_1+l_2+l_3=k}}\nabla^{l_1}G*\nabla^{l_2} h*\nabla^{l_3+1}h),\nabla^kh\rangle\phi^2\,d\mu\\
	&\leq-\sum_{\substack{0\leq l_1,l_2,l_3\leq k\\l_1+l_2+l_3=k}}\int_M\langle\nabla^{l_1}G*\nabla^{l_2} h*\nabla^{l_3+1}h,\nabla^{k+1}h\rangle\phi^2\,d\mu\\&\quad+C
	\sum_{\substack{0\leq l_1,l_2,l_3\leq k\\l_1+l_2+l_3=k}}\int_M|\nabla^{l_1}G||\nabla^{l_2} h||\nabla^{l_3+1}h||\nabla^kh||\nabla\phi|\phi \,d\mu
	\end{align*}
The first of these terms is estimated by
\begin{align*}
-&\sum_{\substack{0\leq l_1,l_2,l_3\leq k\\l_1+l_2+l_3=k}}\int_M\langle\nabla^{l_1}G*\nabla^{l_2} h*\nabla^{l_3+1}h,\nabla^{k+1}h\rangle\phi^2\,d\mu
\\
&=-\int_M \langle G*h*\nabla^{k+1}h,\nabla^{k+1}h\rangle\phi^2\,d\mu\\
&\qquad-\sum_{l=1}^k\sum_{\substack{0\leq l_1,l_2\leq k\\l_1+l_2=k-l+1}}
\int_M\langle\nabla^{l_1}G*\nabla^{l_2} h*\nabla^{l_3+1}h,\nabla^{k+1}h\rangle\phi^2\,d\mu
\\
&\leq (C\left\|h\right\|_{L^{\infty}}+\delta)\int_M |\nabla^{k+1}h|^2\phi^2\,d\mu
+\frac{C}{\delta}\sum_{l=1}^k\sum_{\substack{0\leq l_1,l_2\leq k\\l_1+l_2=k-l+1}}
\int_M|\nabla^{l_1}G|^2|\nabla^{l_2} h|^2|\nabla^{l}h|^2\phi^2\,d\mu,
\end{align*}
where we used the Peter-Paul inequality $ab\leq \delta a^2+\frac{1}{4\delta}b^2$.
For the second of these terms, we estimate
\begin{align*}
C\sum_{\substack{0\leq l_1,l_2,l_3\leq k\\l_1+l_2+l_3=k}}\int_M|\nabla^{l_1}G|&|\nabla^{l_2} h||\nabla^{l_3+1}h||\nabla^kh||\nabla\phi|\phi\, d\mu
=C\int_M |G||h||\nabla^{k+1}h||\nabla^kh||\nabla\phi|\phi\, d\mu\\&\qquad+C\sum_{l=1}^k\sum_{\substack{0\leq l_1,l_2\leq k\\l_1+l_2=k-l+1}}
\int_M|\nabla^{l_1}G||\nabla^{l_2} h||\nabla^{l}h||\nabla^kh||\nabla\phi|\phi \,d\mu\\
&\leq C\left\|h\right\|_{L^{\infty}}\int_M |\nabla^{k+1}h|^2\phi^2\,d\mu+C\int_M |\nabla^kh|^2|\nabla\phi|^2\,d\mu\\
&\qquad + C\sum_{l=1}^k\sum_{\substack{0\leq l_1,l_2\leq k\\l_1+l_2=k-l+1}}
\int_M|\nabla^{l_1}G|^2|\nabla^{l_2} h|^2|\nabla^{l}h|^2\phi^2\,d\mu,
\end{align*}
where we again used the Peter-Paul inequality. Summing up, we get	
	
	\begin{align*}
	\int_M&\langle\nabla(\sum_{\substack{0\leq l_1,l_2,l_3\leq k\\l_1+l_2+l_3=k}}\nabla^{l_1}G*\nabla^{l_2} h*\nabla^{l_3+1}h),\nabla^kh\rangle\phi^2\,d\mu\\
	&\leq (\delta+C\left\|h\right\|_{L^{\infty}})\int_M |\nabla^{k+1}h|^2\phi^2\,d\mu+C\int_M|\nabla^kh|^2|\nabla\phi|^2\,d\mu\\
	&\quad+ C\sum_{l=1}^k\sum_{\substack{0\leq l_1,l_2\leq k\\l_1+l_2=k-l+1}}
\int_M|\nabla^{l_1}G|^2|\nabla^{l_2} h|^2|\nabla^{l}h|^2\phi^2\,d\mu\\
	&\leq  (\delta+C\epsilon)\int_M |\nabla^{k+1}h|^2\phi^2d\mu+C\int_M|\nabla^kh|^2|\nabla\phi|^2\,d\mu+\frac{C}{\delta}\sum_{l=1}^k
	t^{-k+l-1}\int_M |\nabla^lh|^2\phi^2\,d\mu.
	\end{align*}
	Assuming that $\epsilon,\delta>0$ are small enough, we therefore get
	\begin{align*}
	\partial_t\int_M |\nabla^kh|^2\phi^2 \,d\mu&\leq	-\int_M |\nabla^{k+1}h|^2\phi^2d\mu+C_k\int_M |\nabla^k h|^2|\nabla\phi|^2\,d\mu\\
	&\quad+\tilde{C}_k\int_M |h|^2\phi^2\,d\mu
	+ C_k\sum_{l=1}^k t^{-k+l-1}\int_M |\nabla^lh|^2\phi^2\,d\mu.
	%     	+ C(g_0,k,T)\sum_{l=0}^k  t^{l-k}  \int_M |\nabla^{l}h|^2\phi^2 d\mu\\
	%     	&+\tilde{C}(k)\sum_{l=0}^{k}t^{-l-1}\cdot\int_M |\nabla^{k-l}h|^2\phi^2 d\mu
	\end{align*}
	In the following, let $x\in M$ and $\phi_l:M\to [0,1]$, $0\leq l\leq k$ a sequence of cutoff functions with the following properties:
	\begin{align*}
	\phi_l&\equiv 1\qquad \text{ on } B(x,(k+1-l)R),\\
	\phi_l&\equiv 0\qquad \text{ on } M\setminus B(x,(k+2-l)R),\\
	|\nabla \phi_l|&\leq 2/R.
	\end{align*}
	Obviously, we get $\phi_l\leq \phi_{l-1}$ for $1\leq l\leq k$ and, if $R\geq 2$, $|\nabla \phi_l|\leq \phi_{l-1}$ for  $1\leq l\leq k$. We now define a function $F_k:[0,T]\to\mathbb{R}$ as
	\begin{align*}
	F_k(t)=\sum_{l=0}^kA_l\cdot t^l\int_M |\nabla^lh|^2\phi_l^2\,d\mu,
	\end{align*}
	where $A_l$ are some positive constants we will choose later. Then we can compute
	\begin{align*}
	\partial_tF_k&=\sum_{l=1}^kl\cdot A_lt^{l-1}\int_M|\nabla^lh|^2\phi_l^2\,d\mu+\sum_{l=0}^kA_lt^l\partial_t\int_M|\nabla^lh|^2\phi_l^2\,d\mu\\
	&\leq \sum_{l=1}^kl\cdot A_lt^{l-1}\int_M|\nabla^lh|^2\phi_l^2\,d\mu-\sum_{l=0}^k A_lt^{l}\int_M|\nabla^{l+1}h|^2\phi_l^2\,d\mu+\sum_{l=0}^k\tilde{C}_lA_lt^l\int_M |h|^2\phi_l^2\,d\mu\\
	&\quad+\sum_{l=0}^k C_l\cdot A_lt^{l}\int_M|\nabla^lh|^2|\nabla\phi_l|^2\,d\mu
	+\sum_{l=0}^kC_l A_l\sum_{i=1}^lt^{i-1}\int_M |\nabla^ih|^2\phi_l^2\,d\mu\\ 
	%      &=\sum_{l=1}^kl\cdot A_lt^{l-1}\int_M|\nabla^lh|^2(\phi_l)^2d\mu-\sum_{l=1}^k A_{l-1}t^{l-1}\int_M|\nabla^{l}h|^2(\phi_{l-1})^2d\mu\\
	%     &+\sum_{l=0}^k C\cdot A_lt^{l}\int_M|\nabla^lh|^2|\nabla\phi_l|^2d\mu
	%     +C(g_0)\sum_{l=0}^k\sum_{i=l}^kA_it^i\int_M |\nabla^lh|^2(\phi_l)^2d\mu\\
	%    +C(g_0,k,T)&\sum_{l=0}^k\sum_{i=l}^kA_i t^l\int_M |\nabla^lh|^2(\phi_l)^2d\mu+
	%    \tilde{C}(k)\sum_{l=1}^k\sum_{i=l}^kA_it^{l-1}\int_M |\nabla^lh|^2(\phi_i)^2d\mu\\
	&\leq\sum_{l=1}^k[lA_l-A_{l-1}+C_lA_lt+ \sum_{i=l}^kC_iA_i]\cdot t^{l-1}\int_{M}|\nabla^lh|^2\phi_{l-1}^2\,d\mu\\
	&\quad + C_0\cdot A_0\int_M |h|^2||\nabla\phi_0|^2\,d\mu +
	\sum_{l=0}^k\tilde{C}_lA_lt^l\int_M |h|^2\phi_0^2\,d\mu.
	\end{align*}
	Note that we used the properties $\phi_l\leq\phi_{l-1}$ and $|\nabla\phi_l|\leq \phi_{l-1}$ in the above estimate.
	Now if we choose $A_k,A_{k-1},\ldots A_0$ inductively such that
	\begin{align*}
	A_{l-1}\geq lA_l+CA_lt+ \sum_{i=l}^kC_iA_i
	\end{align*}
	for all $t\in[0,T]$,  then
	\begin{align*}
	\partial_t F_k\leq C(g_0,k,T)\int_M |h|^2\phi_0^2\,d\mu\leq C(g_0,k,T)\int_M |h|^2\,d\mu,
	\end{align*}
	so that we get $F_k(t)\leq C(g_0,k,T)\cdot \sup_{t\in[0,T]} \int_M |h|^2\,d\mu$ for all $t\in [0,T]$. The result now follows from letting $R\to\infty$.
\end{proof}
We conclude this section with some very general result due to \cite{Sch-Sch-Sim} giving a criteria for ensuring infinite time existence.

\begin{theo}[Criteria for infinite time existence]\label{delta-max-sol}
	Let $(M^n,g_0)$ be a complete Riemannian manifold such that $\|\Rm(g_0)\|_{L^{\infty}}=:k_0<+\infty$.
	Then there exists a positive constant $\tilde{\delta}=\tilde{\delta}(n,k_0)$ such that the following holds.
	Let $0<\beta<\delta\leq\tilde{\delta}$. Then every metric $g(0)$ $\beta$-close to $g_0$ has a $\delta$-maximal solution $g(t)_{t\in[0,T(g(0)))}$ with $T(g(0))$ positive and $\|g(t)-g_0\|_{L^{\infty}}<\delta$ for all $t\in[0,T(g(0)))$. The solution is $\delta$-maximal in the following sense. Either $T(g(0))=+\infty$ and $\|g(t)-g_0\|_{L^{\infty}}<\delta$ for any nonnegative time $t$ or we can extend $(g(t))_t$ to a solution on $M^n\times[0,T(g(0))+\tau)$, for some positive $\tau=\tau(n,k_0)$ and $\|g(T(g(0)))-g_0\|_{L^{\infty}}=\delta.$
\end{theo}
\subsection{A local decomposition of the space of metrics}\label{deco-met-sec}
In order to prove convergence of a Ricci-DeTurck flow $g(t)$ to a Ricci-flat metric $g_{\infty}$, we have to construct a  family $g_0(t)$ of Ricci-flat reference metrics. For the proof of the main theorem, it is nessecary to construct $g_0(t)$ in such a way that $\partial_t g_0= \textit{o}((g-g_0)^2)$. This section is devoted to this construction.
For this purpose, let $\mathcal{F}$ again be given by
\begin{align*}
\mathcal{F}=\left\{g\in \mathcal{M}\mid -2\Ric(g)+\Li_{V(g,g_0)}g=0 \right\}.
\end{align*}
If $g_0$ satisfies the integrability condition, then there exists an $L^2\cap L^{\infty}$-neighbourhood $\mathcal{U}$ of $g_0$ in the space of metrics such that 
\begin{eqnarray}
\widetilde{\mathcal{F}}:=\mathcal{U}\cap{\mathcal{F}},\label{def-tilde-F}
\end{eqnarray}
 is a manifold and for all $g\in\widetilde{\mathcal{F}}$, the equations $\Ric(g)=0$ and $\Li_{V(g,g_0)}g=0$ hold individually by Proposition \ref{flatsoliton}. Linearization of these two conditions show that the tangent space $T_g\widetilde{\mathcal{F}}$ is given by the kernel of the map
\begin{align*}
L_{g,g_0}h=L_gh-\Li_{\langle h,\Gamma(g)-\Gamma(g_0)\rangle}g,
\end{align*}
where $\langle h,\Gamma(g)-\Gamma(g_0)\rangle^k=g^{ik}g^{lj}h_{kl}(\Gamma(g)_{ij}^k-\Gamma(g_0)_{ij}^k)$ and $L_g$ is the Lichnerowicz Laplacian of $g$.

The next lemma ensures that the kernels $\ker L^*_{g,g_0}$ all have the same dimension when $g$ is an ALE Ricci flat metric sufficiently close to $g_0$:

\begin{lemma}\label{Fred-Prop-Lap-Mod}
	Let $(M^n,g_0)$ be a linearly stable ALE Ricci-flat manifold which is integrable. Furthermore, let $\widetilde{\mathcal{F}}$ be as in (\ref{def-tilde-F}). Then there exists an $L^2\cap L^{\infty}$-neighbourhood $\mathcal{U}$ of $g_0$ in the space of metrics such that $\dim\ker_{L^2} L^*_{g,g_0}=\dim\ker_{L^2} L_{g_0}$ for all $g\in \widetilde{\mathcal{F}}$.
\end{lemma}

\begin{proof}
	First, we claim that elements in the kernel of $L_{g,g_0}$ have 
	decay rate $-(n-1)$. This follows along the lines of the proof of Theorem $2.7$ and we are able to establish \eqref{first_decay} if $h\in\ker_{L^2}L_{g,g_0}$. To improve the decay, we use the special algebraic structure of the operator $L_{g,g_0}$ by considering the divergence and the trace with respect to $g$ of $L_{g,g_0}h$:
	\begin{eqnarray*}
		0&=&\div_g(L_{g,g_0}h)\\
		&=&\Delta_g(\div_gh)-\nabla^g(\div_g(<h,\Gamma(g)-\Gamma(g_0)>))-\Delta_g(<h,\Gamma(g)-\Gamma(g_0)>),\\
		0&=&\Delta_g\tr_gh-2\div_g(<h,\Gamma(g)-\Gamma(g_0)>),
	\end{eqnarray*}
	which implies the following relation:
	\begin{eqnarray*}
		\Delta_g\left(\div_gh-\frac{\nabla^g\tr_gh}{2}-<h,\Gamma(g)-\Gamma(g_0)>\right)=0.
	\end{eqnarray*}
	Since the vector field $\div_gh-\frac{\nabla^g\tr_gh}{2}-<h,\Gamma(g)-\Gamma(g_0)>$ goes to $0$ at infinity, the maximum principle ensures that $$\div_gh-\frac{\nabla^g\tr_gh}{2}-<h,\Gamma(g)-\Gamma(g_0)>=0.$$ 
	An asymptotic expansion of this equation analoguous to \eqref{improved_decay} shows that $h=\textit{O}(r^{-(n-1)}).$
	In particular, the previous claim implies that:
	$$\ker_{L^2}(L_{g,g_0})=\ker_{L^2_{\delta}}(L_{g,g_0})=\ker_{H^k_{\delta}}(L_{g,g_0}),$$ where $\delta\in (-n+1,-n/2]$ is a nonexceptional weight  and $k$ can 
	be any natural number.  
	
	Now, $L_{g_0}=L_{g_0,g_0}$ is Fredholm as a map from $H^k_{\delta}(S^2T^*M)$ to  $H^{k-2}_{\delta-2}(S^2T^*M)$ with Fredholm index $0$. The same holds for $L_g$ with $g\in\mathcal{F}$ in a sufficiently small neighborhood of $g_0$. Observe that $L_{g,g_0}-L_g$ is a bounded operator as a map from $H^k_{\delta}(S^2T^*M)$ to  $H^{k-2}_{\delta-2}(S^2T^*M)$, with arbitrarily small norm operator. Therefore, by the openness of the set of Fredholm operators with respect to the operator norm, $L_{g,g_0}$ has the same index as $L_{g_0,g_0}$, which is $0$. 
	Therefore we get
	\begin{eqnarray*}
		0&=&\dim(\ker_{L^2}(L_{g,g_0}))-\dim(\ker_{L^2}(L^*_{g,g_0}))\\
		&=&\ind_{H^k_{\delta}}(L_{g_0,g_0})\\
		&=&\ind_{H^k_{\delta}}(L_{g,g_0})\\
		&=&\dim(\ker_{H^k_{\delta}}(L_{g,g_0}))-\dim(\ker_{H^k_{\delta}}(L^*_{g,g_0})).
	\end{eqnarray*}
\end{proof}

Now we claim that if $\mathcal{U}$ is small enough, every metric $g\in\mathcal{U}$ can be decomposed uniquely as $g=\bar{g}+h$ where $\bar{g}\in \widetilde{\mathcal{F}}$ and $h\in \overline{L_{\bar{g},g_0}(C^{\infty}_{0}(S^2T^{*}M))}$ (where the closure is taken with respect to $L^2\cap L^{\infty}$). Indeed, this follows from the implicit function theorem applied to the map
\begin{eqnarray*}
	\Phi:\widetilde{\mathcal{F}}\times \overline{L_{g_0}(C^{\infty}_{0}(S^2T^*M))}&\to& \mathcal{M},\\
	(\bar{g},h)&\to& \bar{g}+h-\sum_{i=1}^m(h,e_i(\bar{g}))_{L^2}e_i(\bar{g}),
\end{eqnarray*}
where $\left\{e_1(\bar{g}),\ldots e_m(\bar{g})\right\}$ is an $L^2(\bar{g})$ orthonormal basis of $\ker_{L^2}(L_{\bar{g},g_0}^*)$ which can be chosen to depend smoothly on $\bar{g}$ by Lemma \ref{Fred-Prop-Lap-Mod}.

Let now $(g(t))_{t\in[0,T)}$ be a Ricci-DeTurck flow in $\mathcal{U}$ and $(g_0(t))_{t\in[0,T)}\in \widetilde{\mathcal{F}}$ be the family of Ricci-flat metrics such that $$g(t)-g_0(t)\in\overline{{L_{g_0(t),g_0}(C^{\infty}_{0}(S^2T^{*}M))}}.$$
Writing $h(t)=g(t)-g_0$ and $h_0(t)=g_0(t)-g_0$, we see that $h(t)-h_0(t)=g(t)-g_0(t)$ admits the expansion
\begin{eqnarray*}
	\partial_t h-{L}_{g_0(t),g_0}(h-h_0)&=&R[h-h_0]\\
	&=&F*\nabla(h-h_0)*\nabla(h-h_0)+\nabla(G*(h-h_0)*\nabla(h-h_0)),
\end{eqnarray*}
where the connection is now with respect to $g_0(t)$.

Before stating the next lemma, we need to recall the Hardy inequality for  Riemannian manifolds with nonnegative Ricci curvature and positive asymptotic volume ratio due to Minerbe \cite[Theorem 2.23]{Min-Wei-Sob-Ric-Fla}:
\begin{theo}(Minerbe)\label{Hardy-Ineq}
	Let $(M^n,g)$ be a Riemannian manfold with nonnegative Ricci curvature and Euclidean volume growth, i.e. $$\AVR(g):=\lim_{r\rightarrow+\infty}\frac{\vol_gB_g(x,r)}{r^n}>0,$$ for some (and hence all) $x\in M$. Then,
	\begin{eqnarray*}
		\int_Mr_x^{-2}|\phi|^2d\mu_g\leq C(n,\AVR(g))\int_M|\nabla^g\phi|^2d\mu_g,\quad \forall \phi\in C_0^{\infty}(M),
	\end{eqnarray*}
	where $r_x(y)=d(x,y)$.
\end{theo}
%\textcolor{blue}{The Hardy inequality also holds in a more general setting and also for ALF manifolds, see HaJo Hein, Weighted Sobolev inequalities under lower Ricci curvature bounds, Theorem 1.2. It is maybe better to cite his theorem at this point.}
The  next lemma controls the time derivative of $h_0$ in the $C^k$ topology in terms of the $L^2$ norm of the gradient of $h-h_0$.

\begin{lemma}\label{est_g_0}Let $\mathcal{U}$ be an $L^2\cap L^{\infty}$-neighbourhood of $g_0$ such that the above decomposition holds. Let
	$(g(t))_{t\in[0,T)}$ be a Ricci-DeTurck flow in $\mathcal{U}$ and let $g_0(t)$, $h(t)$, $h_0(t)$ be defined as above for $t\in[0,T)$.
	Then we have the following estimate that holds for $t\in(0,T)$:
	\begin{align*}
	\left\|\partial_t h_0\right\|_{C^k(g_0(t))}\leq C(k)\left\|\nabla^{g_0(t)}(h-h_0)\right\|_{L^2(g_0(t))}^2.
	\end{align*}
\end{lemma}
\begin{proof}
	Let $\left\{e_1(t),\ldots e_m(t)\right\}$ be a family of $L^2(d\mu_{g_0(t)})$-orthonormal bases of $\ker_{L^2}\left({L}_{g_0(t),g_0}^*\right)$. Note that $\partial_te_i(t)$ depends linearly on $\partial_th_0(t)$.
	We can write
	$$h(t)-h_0(t)=k(t)- \sum_{i=1}^m(k(t),e_i(t))_{L^2}\cdot e_i(t)$$
	for some $k(t_0)\in  \overline{L_{g_0(t_0),g_0}(C^{\infty}_{0}(S^2T^{*}M))}=:N$. Note that by this condition on $k(t_0$), we have $(k(t_0),e_i(t_0))_{L^2}=0$ for all $i\in \left\{1,\ldots m\right\}$.
	Therefore, differentiating at time $t_0$ yields
	\begin{align*}
	h'&=h_0'+k'-\sum_{i=1}^m(k',e_i)_{L^2}\cdot e_i -\sum_{i=1}^m(k,e_i')_{L^2}\cdot e_i
	-\sum_{i=1}^m(k*h_0',e_i)_{L^2}\cdot e_i\\
	&=h_0'+k'_N -\sum_{i=1}^m(h-h_0,e_i')_{L^2}\cdot e_i
	-\sum_{i=1}^m((h-h_0)*h_0',e_i)_{L^2}\cdot e_i\\
	&=:h_0'+k'_N +A(h-h_0,h_0'),
	\end{align*}
	where $A$ depends linearly on both entries.
	Let us split this expression into
	$A=A_{\widetilde{\mathcal{F}}}+A_{N}$
	according to the decomposition $T_{g_0(t_0)}\widetilde{\mathcal{F}}\oplus N$.
	If we also split $h'=h'_{\widetilde{\mathcal{F}}}+h'_{N}$, we get
	\begin{align*}
	\begin{pmatrix}
	h'_{\widetilde{\mathcal{F}}} \\ h'_{N}\end{pmatrix}=
	\begin{pmatrix}
	id_{T_{g_0(t_0)}\widetilde{\mathcal{F}}}+A_{\widetilde{\mathcal{F}}}(h-h_0,.) & 0 \\
	A_{N}(h-h_0,.) & id_{\bar{N}}
	\end{pmatrix}
	\cdot \begin{pmatrix}
	h_0' \\ k' \end{pmatrix}
	\end{align*}
	By inverting, we conclude that
	\begin{align*}
	h'_0=(id_{T_{g_0(t_0)}\widetilde{\mathcal{F}}}+A_{\widetilde{\mathcal{F}}}(h-h_0,.))^{-1}h'_{\widetilde{\mathcal{F}}}.
	\end{align*}
	Note that the orthogonal projection $\Pi:T_{g_0(t_0)}\widetilde{\mathcal{F}}\to \ker_{L^2}\left({L}_{g_0(t_0),g_0}^*\right)$ is an isomorphism. 
	Because $\partial_th={L}_{g_0(t),g_0}(h-h_0)+R[h-h_0]$ and by elliptic regularity,
	\begin{align*}
	&\left\|h_0'(t)\right\|_{C^k(g_0(t))}\leq C\left\|(h'(t))_{\widetilde{\mathcal{F}}}\right\|_{L^2(g_0(t))}\\
	%&=C\left\|(R[h-h_0])_{\widetilde{\mathcal{F}}}\right\|_{L^2(g_0(t))}\\
	&\leq C\left\|(R[h-h_0])_{\ker_{L^2}\left({L}_{g_0(t_0),g_0}^*\right)}\right\|_{L^2(g_0(t))}\\
	&\leq C\sum_{i=1}^m| (F*\nabla^{g_0(t)}(h-h_0)*\nabla^{g_0(t)}(h-h_0)+\nabla^{g_0(t)}(G*(h-h_0)*\nabla^{g_0(t)}(h-h_0)),e_i)_{L^2(g_0(t))}|\\
	&\leq C\left\|\nabla^{g_0(t)}(h-h_0)\right\|_{L^2(g_0(t))}^2+\sum_{i=1}^m|(G*r^{-1}\cdot(h-h_0)*\nabla^{g_0(t)}(h-h_0),r\nabla^{g_0(t)} e_i)_{L^2(g_0(t))}|\\
	&\leq C\left\|\nabla^{g_0(t)}(h-h_0)\right\|_{L^2(g_0(t))}^2.
	\end{align*}
	%	Here, the covariant derivatives here are taken with respect to $g_0(t)$.
	We used the Hardy inequality (Theorem \ref{Hardy-Ineq}) and the fact that $r\nabla^{g_0(t)} e_i$ is bounded by elliptic regularity in the last step.
\end{proof}
Before proving theorem \ref{main-theo}, we start by recalling a result by Devyver \cite[Definition 6]{Dev-Gau-Est} adapted to our context:
\begin{theo}[Strong positivity of $L_{g_0}$]\label{Dev-Str-Pos}
	Let $(M^n,g_0)$ be an ALE Ricci-flat space that is linearly stable. Then the restriction of $-L_{g_0}$ to the orthogonal of $\ker_{L^2}(L_{g_0})$ is strongly positive, i.e. there exists some positive $\alpha_{g_0}\in(0,1]$ such that
	\begin{eqnarray*}
		\alpha_{g_0}(-\Delta_{g_0}h,h)_{L^2(g_0)}\leq (-L_{g_0}h,h)_{L^2(g_0)}\quad \forall h\in L_{g_0}(C^{\infty}_{0}(S^2T^{*}M)).
	\end{eqnarray*}
	\begin{proof}[Sketch of proof] The proof is as in \cite{Dev-Gau-Est} with some minor modifications we point out here.
		Write $-L_{g_0}=-\Delta_{g_0}+R_+-R_-$ where $R_+$ and $R_-$ correspond to the positive (resp.\ non-positive) eigenvalues of $Rm(g_0)*$. Let $H=-\Delta_{g_0}+R_+$ and $A:L^2(S^2T^*M)\to L^2(S^2T^*M)$ be defined by $A=H^{-1/2}R_-H^{-1/2}$. The operator $A$ is compact \cite[Corollary 1.3]{Car-Coh}
		Because $-L_{g_0}$ is nonnegative, all eigenvalues of $A$ lie in $[0,1]$. It can be shown that $H^{1/2}$ maps $\ker_{L^2}(L_{g_0})$ isomorphically to $\ker_{L^2}(1-A)$. As $A$ is a compact operator, we get the condition
		\begin{align*}
		(Ah,h)_{L^2(g_0)}\leq (1-\epsilon)\left\|h\right\|_{L^2(g_0)}^2,\quad \forall h\in H^{1/2}(C^{\infty}_{0}(S^2T^*M)\cap \ker_{L^2}(L_{g_0})^{\perp}),
		\end{align*}
		for some $\epsilon>0$
		which in turn is equivalent to
		\begin{align*}
		(R_-h,h)_{L^2(g_0)}\leq (1-\epsilon)(Hh,h)_{L^2(g_0)},\quad \forall h\in C^{\infty}_{0}(S^2T^*M)\cap \ker_{L^2}(L_{g_0})^{\perp}.
		\end{align*}
		%The assertion follows immediately. For details, see \cite[p. 37-39]{Dev-Gau-Est}. \textcolor{blue}{Observe that the same proof also works in the ALF case. One has just to recall that ALF spaces satisfy the Sobolev inequality (Saloff-Coste, A note on Poincare, Harnack and Sobolev inequalities) and that the curvature decays faster than quadratically.}
	\end{proof}
	%\textcolor{red}{Hardy inequality!}
\end{theo}
\begin{theo}\label{theo-unif-bound-strict-pos}
	Let $(M^n,g_0)$ be a linearly stable ALE Ricci-flat manifold which is integrable. Furthermore, let $\widetilde{\mathcal{F}}$ be as in (\ref{def-tilde-F}). Then there exists a constant $\alpha_{g_0}>0$ such that 
	\begin{eqnarray*}
		(-{L}_{g,g_0}h,h)_{L^2(g)}\geq \alpha_{g_0} \left\|\nabla^{g} h\right\|^2_{L^2(g)},
	\end{eqnarray*}
	for all $g\in \widetilde{\mathcal{F}}$ and $h\in L_{g_0(t),g_0}(C^{\infty}_{0}(S^2T^{*}M))$ provided that $\widetilde{\mathcal{F}}$ is chosen small enough.
\end{theo}
\begin{proof}
	By Theorem \ref{Dev-Str-Pos}, there exists a constant $\alpha_0>0$ such that
	\begin{eqnarray*}
		(-L_{g_0}h,h)_{L^2({g_0})}\geq \alpha_0 \left\|\nabla^{g_0} h\right\|_{L^2({g_0})},
	\end{eqnarray*} for any compactly supported $h\in \ker_{L^2}(L_{g_0})^{\perp}$.
	Now by Taylor expansion, with $k=g-g_0$
	\begin{align*}
	&(-{L}_{g,g_0}h,h)_{L^2({g})}=	(-L_{g_0}h,h)_{L^2({g_0})}-\int_0^1\frac{d}{dt}({L}_{g_0+tk,g_0}h,h)_{L^2(g_0+tk)}dt\\
	&\geq \alpha_0  \left\|\nabla^{g_0} h\right\|_{L^2({g_0})}^2+\int_0^1 [(\nabla^{g_0,2}h*k+\nabla^{g_0} h*\nabla^{g_0} k+h*\nabla^{g_0,2} k)*h+Rm*h*h*k] \,d\mu_{g_0+tk}dt\\
	&= \alpha_0  \left\|\nabla^{g_0} h\right\|_{L^2({g_0})}^2+\int_0^1 [\nabla^{g_0} h*h*\nabla^{g_0} k+\nabla^{g_0} h*\nabla^{g_0} h* k+Rm*h*h*k] \,d\mu_{g_0+tk}dt\\
	&\geq  \alpha_0  \left\|\nabla^{g_0} h\right\|_{L^2({g_0})}^2- C \left\|\nabla^{g_0} h\right\|_{L^2({g_0})}^2\left\| k\right\|_{L^{\infty}(g_0)}- C \left\|\nabla^{g_0} h\right\|_{L^2({g_0})}\left\|h*\nabla^{g_0} k\right\|_{L^2({g_0})}\\
	&\geq  \alpha_0  \left\|\nabla^{g_0} h\right\|_{L^2({g_0})}^2- C \left\|\nabla^{g_0} h\right\|_{L^2({g_0})}^2\left\| k\right\|_{L^{\infty}(g_0)}^{\sigma},
	\end{align*}
	for some $\sigma\in (0,1)$.
	%	Here we used the fact that all $(g_0+tk)_{t\in[0,1]}$-norms are equivalent. 
	To justify the last inequality, we use elliptic regularity and Sobolev embedding as in the proof of Theorem \ref{ell-reg-prop} to obtain
	\begin{align*}
	\left\|r\nabla^{g_0} k\right\|_{L^{\infty}({g_0})}\leq\left\|k\right\|_{C^{1,\alpha}_{0}({g_0})}\leq C\left\|k\right\|_{W^{2,p}_{0}({g_0})}\leq  C\left\|k\right\|_{L^{p}({g_0})}\leq C\left\|k\right\|_{L^{2}({g_0})}^{1-\sigma}\left\|k\right\|_{L^{\infty}({g_0})}^{\sigma},
	\end{align*}
	with $\sigma=1-\frac{2}{p}$ and $p>n$. This can be combined with the Hardy inequality (\ref{Hardy-Ineq}) to obtain
	\begin{align*}
	\left\|h*\nabla^{g_0} k\right\|_{L^2({g_0})}\leq &	\left\|r^{-1}h\right\|_{L^2({g_0})}	\left\|r\nabla^{g_0} k\right\|_{L^{\infty}({g_0})}
	\leq  C\left\|\nabla^{g_0} h\right\|_{L^2({g_0})}\left\|k\right\|_{L^{\infty}(g_0)}^{\sigma},
	\end{align*}
	which yields the estimate of the theorem for $h\in \ker_{L^2}(L_{g_0})^{\perp}$, provided that $\left\| k\right\|_{L^{\infty}(g_0)}$ is small enough. To pass to $h\in \ker_{L^2}({L}_{g,g_0}^*)^{\perp}$, we note that an isomorphism between $\ker_{L^2}(L_{g_0})^{\perp}$ and $ \ker_{L^2}({L}_{g,g_0}^*)^{\perp}$ is given by 
	$\Phi_g:h\mapsto h-\sum_i(h,e_i(g))_{L^2(g)}\cdot e_i(g)$ where the tensors $e_i(g)$ are an orthonormal basis of $\ker_{L^2}({L}_{g,g_0}^*)$. At first, we have
\begin{align*}
\nabla^g\Phi_g(h)=\nabla^gh-\sum_i(h,e_i(g))_{L^2(g)}\cdot \nabla^g e_i(g),
	\end{align*}
	from which we conclude, using integration by parts
	\begin{equation}\begin{split}\label{nabla-est}
	\left\|\nabla^g \Phi_g(h)\right\|^2_{L^2(g)}&=\left\|\nabla^g h\right\|^2_{L^2(g)}
	-2(h,e_i)_{L^2(g)}(\Delta_g e_i,h)_{L^2(g)}+(h,e_i)_{L^2(g)}^2(\Delta_g e_i,e_i)_{L^2(g)}\\
	&\leq C\left\|\nabla^g h\right\|^2_{L^2(g)}
	\leq C\left\|\nabla^{g_0} h\right\|^2_{L^2(g_0)}.
	\end{split}
	\end{equation}
	The first inequality here can be proven as follows: Because $e_i=\textit{O}(r^{-n+1})$ as $r\to\infty$ (c.f.\ Theorem \ref{ell-reg-prop}), we have
	\begin{align*}
(h,e_i)_{L^2(g)}\leq \left\| r^{-1}h\right\|_{L^2(g)}\left\|r e_i\right\|_{L^2(g)}\leq C\left\| \nabla^gh\right\|_{L^2(g)}
\end{align*}
due to the Hardy inequality
The same argument also yields
\begin{align*}
(\Delta_g e_i,h)_{L^2(g)}\leq  C\left\| \nabla^gh\right\|_{L^2(g)},
\end{align*}
because $\Delta_ge_i=\textit{O}(r^{-n-1})$ as $r\to\infty$ (c.f.\ Theorem \ref{ell-reg-prop} again). 	This justifies the first inequality from above. 
For the second inequality, we write
\begin{align*}
\nabla^gh=\nabla^{g_0}h+(\Gamma(g_0)-\Gamma(g))*h,
\end{align*}
and $(\Gamma(g_0)-\Gamma(g))=\textit{O}(r^{-n})$, see Theorem \ref{ell-reg-prop}.
Again by the Hardy inequality,
\begin{align*}
\left\|\nabla^gh\right\|_{L^2(g)}^2&\leq C(\left\|\nabla^{g_0}h\right\|_{L^2(g)}^2+\left\|(\Gamma(g_0)-\Gamma(g))*h\right\|_{L^2(g)}^2)\\
&\leq C(\left\|\nabla^{g_0}h\right\|_{L^2(g_0)}^2+\left\|r^{-1}h\right\|_{L^2(g_0)}^2)\leq C \left\|\nabla^{g_0}h\right\|_{L^2(g_0)}^2.
\end{align*}
Here, we use the smallness of $\left\|g-g_0\right\|_{L^{\infty}(g_0)}$ to compare the corresponding $L^2$-norms. 
%
%	 let $\left\{f_i\right\}_i$  be an orthonormal basis of $\ker_{L^2}(L_{g_0})$, then by the triangular inequality,
%	\begin{align*} |(h,e_i)_{L^2(g)}|&\leq|(h,e_i)_{L^2(g_0)}|+|(h,e_i)_{L^2(g_0)}-(h,e_i)_{L^2(g)}|.
%	\end{align*}
%	For the first of these terms, we have
%	\begin{align*}
%	|(h,e_i)_{L^2(g_0)}|&= C|(h,e_i-f_i)_{L^2(g_0)}|\\&\leq C\left\|h\right\|_{L^{2n/(n-2)}(g)}\cdot \left\|e_i-f_i\right\|_{L^{n/2}(g)}\leq  \delta\left\| \nabla^g h\right\|_{L^2(g)},
%	\end{align*}
%	where we used the Sobolev inequality and the fact that a basis of $\ker_{L^2}({L}_{g,g_0}^*)$ can be chosen to depend smoothly on $g$. To handle the second term, we use Taylor expansion and we obtain
%	\begin{align*}
%	|(h,e_i)_{L^2(g_0)}-(h,e_i)_{L^2(g)}|&\leq C \int_M |k|\cdot |h|d\mu_g \leq \left\|h\right\|_{L^{2n/(n-2)}(g_0)}\cdot \left\|k\right\|_{L^{n/2}(g_0)}\\&\leq C\left\| \nabla^{g_0} h\right\|_{L^2(g_0)}\cdot  \left\|k\right\|_{L^{\infty}(g_0)}^{\sigma}\leq \delta\left\| \nabla^{g_0} h\right\|_{L^2(g_0)}.
%	\end{align*}
	To finish the proof, it remains to show that the inequality $$(-{L}_{g,g_0} \Phi_g(h),\Phi_g(h))_{L^2(g)}\geq C\cdot (-{L}_{g,g_0} h,h)_{L^2(g)}$$ holds for some constant $C>0$. We compute
	\begin{align*}
	(-{L}_{g,g_0}\Phi_g(h),\Phi_g(h))_{L^2(g)}	
	&=(-{L}_{g,g_0}h,h)_{L^2(g)}+\sum_i (h,e_i)_{L^2(g)}(-{L}_{g,g_0}e_i,h)_{L^2(g)}\\
	&\geq(-{L}_{g,g_0}h,h)_{L^2(g)}-\sum_i \delta' \left\| \nabla^g h\right\|_{L^2(g)}\left\|{L}_{g,g_0}e_i\right\|_{L^{n/2}(g)}\left\|h\right\|_{L^{2n/n-2}(g)}
	\\
	&\geq (-{L}_{g,g_0}h,h)_{L^2(g)}- \delta \left\| \nabla^g h\right\|_{L^2(g)}^2\\
	&\geq (1-C\delta)(-{L}_{g,g_0}h,h)_{L^2(g)},
	\end{align*}
	where we also used the Sobolev inequality and elliptic regularity.
\end{proof}
\subsection{Existence for all time and convergence}\label{Exi-conv}
\begin{prop}\label{L2-bound}
	Let $(M,g_0)$ be a linearly stable ALE Ricci-flat manifold which satisfies the integrability condition.
	Then there exists an $\epsilon>0$ with the following property: If $(g(t))_{t\in[0,T]}$ is a Ricci-DeTurck flow and $T$ a time such that $\left\|g(t)-g_0\right\|_{L^{\infty}}<\epsilon$ for all $t\in [0,T]$, then there exists a constant such that the evolution inequality
	\[ \frac{d}{dt}\left\|h-h_0\right\|_{L^2(g_0(t))}^2+C\left\|\nabla^{g_0(t)}(h-h_0)\right\|_{L^2(g_0(t))}^2\leq0,
	\]
	holds.
\end{prop}
\begin{proof}
	We know that
	\begin{eqnarray*}
		\partial_t h-{L}_{g_0(t),g_0}(h-h_0)=R[h-h_0],
	\end{eqnarray*}
	where
	\begin{eqnarray*}
		R[h-h_0]=F*\nabla^{g_0(t)}(h-h_0)*\nabla^{g_0(t)}(h-h_0)+\nabla^{g_0(t)}(G*(h-h_0)*\nabla^{g_0(t)}(h-h_0)).
	\end{eqnarray*}
	Thanks to Theorem \ref{theo-unif-bound-strict-pos} and Lemma \ref{est_g_0},
	\begin{align*}
	\partial_t&\|h-h_0\|_{L^2(g_0(t))}^2=2({L}_{g_0(t),g_0}(h-h_0),h-h_0)_{L^2(g_0(t))}+(R[h-h_0],h-h_0)_{L^2(g_0(t))}  \\& +(h-h_0,\partial_t h_0(t))_{L^2(g_0(t))}+\int_M (h-h_0)*(h-h_0)*\partial_t h_0(t)d\mu_{g_0(t)}\\
	&\leq -2\alpha_{g_0} \left\|\nabla^{g_0(t)}(h-h_0)\right\|_{L^2(g_0(t))}^2+C\left\|(h-h_0)\right\|_{L^{\infty}(g_0(t))}\left\|\nabla^{g_0(t)}(h-h_0)\right\|_{L^2(g_0(t))}^2\\
	& \quad+ 	\left\|\partial_th_0\right\|_{L^{2}(g_0(t))}\left\|h-h_0\right\|_{L^2(g_0(t))}\\
	&\leq (-2\alpha_{g_0}+C\cdot \epsilon )\left\|\nabla^{g_0(t)}(h-h_0)\right\|_{L^2(g_0(t))}^2,
	\end{align*}
	which proves the desired estimate.
\end{proof}

\begin{proof}[Proof of Theorem \ref{main-theo}]
	Let  $\epsilon>0$ be so small that Proposition \ref{L2-bound} holds, provided that $\left\|h\right\|_{L^{\infty}(g_0)}<\epsilon$.
	% Let now $$E_{k_0}(t)=\sum_{k=0}^{k_0}\beta_k((-L)^k(h-h_0),(h-h_0))_{L^2}\qquad t\geq 1.$$This expression is finite for positive time due to Lemma \ref{Hk-estimate}. Note also that this expression is equivalent to the $H_{k_0}$-Norm due to elliptic regularity.
	By Lemma \ref{Ck-estimate}, we can find $\delta_1>0$ so small that the above $\epsilon$-bound holds as long as $h\in\mathcal{B}_{L^2\cap L^{\infty}}(0,\delta_1)$.
	Now let  $\delta_2>0$ be so small that $h(1)\in 
	\mathcal{B}_{L^2\cap L^{\infty}}(0,\delta_1)$ if $h(0)\in \mathcal{B}_{L^2\cap L^{\infty}}(0,\delta_2)$ where $\delta_2=\delta_2(\epsilon)$ will be chosen below.	
	Suppose that $T_{max}\geq 1$ is the first time where $h(t)$ leaves $\mathcal{B}_{L^2\cap L^{\infty}}(0,\epsilon)$. 
	By Lemma \ref{est_g_0}, Proposition \ref{L2-bound} and elliptic regularity,
	\begin{align*}
	\left\|h_0(T_{max})\right\|_{L^{2}(g_0)}+&\left\|h_0(T_{max})\right\|_{L^{\infty}(g_0)}\leq C\int_1^{T_{max}}\left\|\partial_th_0(t)\right\|_{L^{2}(g_0(t))}dt\\&
	\leq C\int_1^{T_{max}}\left\|\nabla^{g_0(t)} (h(t)-h_0(t))\right\|_{L^{2}(g_0(t))}^2dt\\
	&\leq C\left\|h(1)-h_0(1)\right\|_{L^{2}(g_0)}^2\leq C\left\|h(1)\right\|_{L^{2}(g_0)}^2\leq C\cdot (\delta_1)^2.
	\end{align*}	
	Furthermore by Proposition \ref{L2-bound},
	\begin{align*}
	\left\|h(T_{max})-h_0(T_{max})\right\|_{L^{2}(g_0)}\leq 	\left\|h(1)-h_0(1)\right\|_{L^{2}(g_0)}\leq C\cdot \delta_1.
	\end{align*}
	Again by Proposition \ref{L2-bound}, Lemma \ref{Ck-estimate} and interpolation,
	\begin{align*}
	&\left\|h(T_{max})-h_0(T_{max})\right\|_{L^{\infty}(g_0)}\\&\qquad\leq 	C\left\|\nabla^{g_0}(h(T_{max})-h_0(T_{max}))\right\|_{L^{\infty}(g_0)}^{1-\alpha}\cdot \left\|h(T_{max})-h_0(T_{max})\right\|_{L^{2}(g_0)}^{\alpha}\\
	&\qquad\leq C\cdot \epsilon^{1-\alpha}\left\|h(1)-h_0(1)\right\|_{L^{2}(g_0)}^{\alpha}\leq C\cdot \epsilon^{1-\alpha}(\delta_1)^{\alpha},
	\end{align*}
	with $\alpha=\frac{n}{n+2}$.
	%	Due to Proposition \ref{energy-monotonicity}, 
	%	\begin{align*}
	%		\left\|h(T_{max})\right\|_{L^{\infty}}\leq C_s E_{k_0}(T_{max})\leq C_s E_{k_0}(1)<\delta_1/4.
	%	\end{align*}
	By the triangle inequality,
	\begin{align*}
	\left\|h(T_{max})\right\|_{L^{2}(g_0)}+\left\|h(T_{max})\right\|_{L^{\infty}(g_0)}&\leq C\cdot (\delta_1)^2+C\cdot \delta_1+ C\cdot \epsilon^{1-\alpha}(\delta_1)^{\alpha}\leq \epsilon/2,
	\end{align*}
	provided that $\delta_1>0$ was chosen small enough. We now have proven that such a $T_{max}$ can not exist and that $h(t)\in \mathcal{U}$ for all $T>0$.
	
	Moreover, because
	\begin{align*}
	\int_1^{\infty}\left\|\partial_th_0(t)\right\|_{L^{2}(g_0(t))}dt
	&\leq C\int_1^{\infty}\left\|\nabla^{g_0(t)} (h(t)-h_0(t))\right\|_{L^{2}(g_0(t))}^2dt\\&\leq C\left\|h(1)-h_0(1)\right\|_{L^{2}(g_0)}^2<\infty,
	\end{align*}
	and since
	\begin{align*}
	\left|\partial_t \left\|\nabla^{g_0(t)}(h(t)-h_0(t))\right\|^2_{L^2(g_0(t))}\right|\leq C \left\|h(t)-h_0(t)\right\|^2_{H^{3}(g_0)}\leq C,
	\end{align*}
	by Lemma \ref{Hk-estimate},	we get 
	$$\limsup_{t\rightarrow+\infty}\left\|\partial_th_0(t)\right\|_{L^{2}(g_0(t))}\leq \limsup_{t\rightarrow+\infty}\left\|\nabla^{g_0(t)} (h(t)-h_0(t))\right\|_{L^{2}(g_0(t))}^2= 0,$$ and $h_0(t)$ converges in $L^2$ to some limit $h_0(\infty)$. Due to elliptic regularity, $(h_0(t))_{t\geq 0}$ converges to  $h_0(\infty)$ as $t$ goes to $+\infty$ with respect to all Sobolev norms. We are going to show now that $(g(t))_{t\geq 0}$ converges to $g_0+h_0(\infty)=: g_{\infty}$ as $t$ goes to $+\infty$ with respect to all $W^{k,p}$-norms with $p>2$.
	For this purpose, it suffices to show that $(h(t)-h_0(t))_{t\geq 0}$ converges to $0$ as $t$ goes to $+\infty$ with respect to all these norms. At first, by the Euclidean Sobolev inequality,
	\begin{align*}
	\left\|h-h_0\right\|_{L^{\frac{2n}{n-2}}(g_0)}\leq C\left\|\nabla^{g_0} (h-h_0)\right\|_{L^2(g_0)}\to0, \quad\mbox{as $t\rightarrow+\infty$,}
	\end{align*}
	which implies that $\lim_{t\rightarrow+\infty}h-h_0=0$ in $L^p$ for all $p\in(2,\infty)$ by interpolation and due to smallness in $L^2\cap L^{\infty}$.
	Moreover, for $j\in\mathbb{N}$ arbitrary, by interpolation inequalities,
	\begin{align*}
	\left\|\nabla^{g_0,j}(h-h_0)\right\|_{L^p(g_0)}\leq C\left\|\nabla^{g_0,m}(h-h_0)\right\|_{L^{\infty}(g_0)}^{\alpha}\left\|h-h_0\right\|_{L^p(g_0)}^{1-\alpha}\leq C\left\|h-h_0\right\|_{L^p(g_0)}^{1-\alpha}\to0,
	\end{align*}
	as $t\rightarrow+\infty$ with $\alpha=\frac{jp}{mp+n}$ and $m\in\mathbb{N}$ so large that $\alpha<1$.
	Due to Sobolev embedding, convergence also holds for $p=\infty$.
	
\end{proof}

\section{Nash-Moser Iteration at infinity}\label{Nash-Moser-Sec}

In this section, we prove a decay of the $L^{\infty}$ norm of the difference between an immortal solution to the Ricci-DeTurck flow with gauge an ALE Ricci-flat metric $g_0$ and the metric $g_0$ itself. More precisely, one has the following theorem:

\begin{theo}\label{theo-Nas-Moser-Iteration}
Let $(M^n,g_0)$ be a complete Riemannian manifold endowed with a Ricci flat metric with quadratic curvature decay at infinity, i.e.
\begin{eqnarray*}
\Ric(g_0)= 0,\quad |\Rm(g_0)|_{g_0}(x)\leq \frac{C}{1+d_{g_0}^2(x_0,x)},\quad x\in M,\quad\AVR(g_0)>0,
\end{eqnarray*}
for some positive constant $C$ and some point $x_0\in M$.
Let $(M^n,g(t))_{t\geq 0}$, be an immortal solution to the Ricci-DeTurck flow with respect to the background metric $g_0$ such that $$\sup_{t\geq 0}\|g(t)-g_0\|_{L^{\infty}(M)}\leq \epsilon,$$ for some positive universal $\epsilon$. Then, for any positive time $t$, and radius $r<\sqrt{t}$,
\begin{eqnarray*}
&&\sup_{P(x_0,t,r/2)}|g(t)-g_0|^2\leq \frac{C(n,g_0,\epsilon)}{r^{n+2}}\int_{P(x_0,t,r)}\arrowvert g(t)-g_0\arrowvert_{g_0}^2(y,s)d\mu_{g_0}(y)ds,\\
&&P(x_0,t,r):=(M\setminus B_{g_0}(x_0,r))\times (t-r^2,t].
\end{eqnarray*}
In particular, if $\sup_{t\geq 0}\|g(t)-g_0\|_{L^2(M)}\leq C<+\infty$, then 
\begin{eqnarray*}
\|g(t)-g_0\|_{L^{\infty}(M\setminus B_{g_0}(x_0,\sqrt{t}))}\leq C(n,g_0,\epsilon)\frac{\sup_{t\geq 0}\|g(t)-g_0\|_{L^2(M)}}{t^{\frac{n}{4}}},\quad t>0.
\end{eqnarray*}

\end{theo}

\begin{rk}
Notice that $P(x_0,t,r)$ is the parabolic neighborhood of a point at infinity on the manifold $M$. 
\end{rk}
Before starting the proof of this theorem, we remark that by combining Theorem \ref{main-theo} and Theorem \ref{theo-Nas-Moser-Iteration} leads to Theorem \ref{main-theo-bis}.

\begin{proof}[Proof of Theorem \ref{theo-Nas-Moser-Iteration}]
Recall that $(g(t))_{ t\geq 0}$ (or similarly $h(t):=g(t)-g_0$) satisfies the following partial differential equation :
\begin{eqnarray*}
\partial_th&=&g^{-1}\ast\nabla^{g_0,2}h+\Rm(g_0)\ast h+g^{-1}\ast g^{-1}\ast\nabla^{g_0}h\ast\nabla^{g_0}h,\\
g^{-1}\ast\nabla^{g_0,2}h&:=&g^{ij}\nabla^{g_0,2}_{ij}h.
\end{eqnarray*}
 In particular,
 \begin{eqnarray*}
&&\partial_t|h|_{g_0}^2\leq g^{-1}\ast\nabla^{g_0,2}|h|_{g_0}^2-2g^{-1}(\nabla^{g_0}h,\nabla^{g_0}h)+\RP(g_0)|h|_{g_0}^2+c(n)|h|_{g_0}|\nabla^{g_0}h|_{g_0}^2,\\
&&g^{-1}(\nabla^{g_0}h,\nabla^{g_0}h):=g^{ij}\nabla^{g_0}_ih\nabla^{g_0}_jh,\quad \RP(g_0):=c(n,\epsilon)|\Rm(g_0)|_{g_0}.
\end{eqnarray*}
Since $\|h(t)\|_{L^{\infty}(M)}\leq \epsilon$ for any positive time $t$ :
\begin{eqnarray*}
&&\partial_t|h|_{g_0}^2\leq g^{-1}\ast\nabla^{g_0,2}|h|_{g_0}^2-|\nabla^{g_0}h|_{g_0}^2+\RP(g_0)|h|_{g_0}^2,\\
\end{eqnarray*}
if $\epsilon\leq \epsilon(n)$.
Define $u:=|h|_{g_0}^2$ and multiply the previous differential inequality by $pu^{p-1}$ for some real $p\geq 2$ to get :
\begin{eqnarray*}
\partial_tu^p&=&pu^{p-1}\partial_tu\\
&\leq& pu^{p-1}g^{-1}\ast\nabla^{g_0,2}u-pu^{p-1}|\nabla^{g_0}h|_{g_0}^2+p\RP(g_0)u^p\\
&\leq&g^{-1}\ast\nabla^{g_0,2}u^p-pg^{-1}(\nabla^{g_0}u^{p-1},\nabla^{g_0}u)-pu^{p-1}|\nabla^{g_0}h|_{g_0}^2+p\RP(g_0)u^p.\\
\end{eqnarray*}

Take any smooth space-time cutoff function $\psi$ and multiply the previous differential inequality by $\psi^2u^p$ and integrate by parts as follows : 
 \begin{eqnarray}
&&\int_{t- r^2}^{t'}\int_M\psi^2u^p\left[-g^{-1}\ast\nabla^{g_0,2}u^p+pg^{-1}(\nabla^{g_0}u^{p-1},\nabla^{g_0}u)+pu^{p-1}|\nabla^{g_0}h|_{g_0}^2\right]d\mu_{g_0}ds\label{inequ-nas-mos-1}\\
&&\leq \int_{t-r^2}^{t'}\int_M-\psi^2u^p\partial_su^p+p\RP(g_0)\psi^2u^{2p}d\mu_{g_0}ds,\label{inequ-nas-mos-2}
\end{eqnarray}
for any $t'\in(t-r^2,t].$

Now, by integrating by parts once and using the pointwise Young inequality :
\begin{eqnarray*}
-\int_M\psi^2u^pg^{-1}\ast\nabla^{g_0,2}u^pd\mu_{g_0}&=&\int_M\nabla^{g_0}_i(g^{ij}\psi^2u^p)\nabla^{g_0}_ju^pd\mu_{g_0}\\
&=&\int_Mg^{-1}(\nabla^{g_0}(\psi u^p),\nabla^{g_0}(\psi u^p))-g^{-1}(\nabla^{g_0}\psi,\nabla^{g_0}(\psi u^p))u^pd\mu_{g_0}\\
&&+\int_Mg^{-1}(u^p\nabla^{g_0}\psi,\psi\nabla^{g_0}u^p)d\mu_{g_0}ds\\
&&+\int_M\div_{g_0}(g^{-1})(\psi^2u^p\nabla^{g_0}u^p)d\mu_{g_0}ds\\
&\geq&\frac{1}{2}\int_M|\nabla^{g_0}(\psi u^p)|_{g_0}^2d\mu_{g_0}\\
&&-c\int_M|\nabla^{g_0}\psi|^2_{g_0}u^{2p}+pu^{2p-1}\psi|\nabla^{g_0}h|_{g_0}\psi|\nabla^{g_0}u|_{g_0}d\mu_{g_0}\\
&\geq&\frac{1}{2}\int_M|\nabla^{g_0}(\psi u^p)|_{g_0}^2d\mu_{g_0}-c\int_M|\nabla^{g_0}\psi|^2_{g_0}u^{2p}d\mu_{g_0}\\
&&-\frac{p}{2}\int_M\left(\psi^2u^{2p-1}|\nabla^{g_0}h|_{g_0}^2+c u^{2p-1}\psi^2|\nabla^{g_0}u|^2_{g_0}\right)d\mu_{g_0},
\end{eqnarray*}
for some universal positive constant $c$, and where we used the smallness of $\|h(t)\|_{L^{\infty}(M)}$ for all time $t$. Going back to (\ref{inequ-nas-mos-1}) and (\ref{inequ-nas-mos-2}), one gets by absorbing terms appropriately :
\begin{eqnarray*}
\frac{1}{2}\int_{t-r^2}^{t'}\int_M|\nabla^{g_0}(\psi u^p)|_{g_0}^2d\mu_{g_0}ds&\leq& \int_{t-r^2}^{t'}\int_M-\psi^2u^p\partial_su^p+p\RP(g_0)\psi^2u^{2p}d\mu_{g_0}ds \\
&&+c\int_{t-r^2}^{t'}\int_M|\nabla^{g_0}\psi|^2_{g_0}u^{2p}d\mu_{g_0}ds.
\end{eqnarray*}
Finally, by integrating by parts with respect to time :
\begin{eqnarray*}
&&\int_M(u^{2p}\psi^2)(t')d\mu_{g_0}+\int_{t-r^2}^{t'}\int_M|\nabla^{g_0}(\psi u^p)|_{g_0}^2d\mu_{g_0}ds\\&&\qquad\leq 2\int_{t-r^2}^{t'}\int_Mp\RP(g_0)\psi^2u^{2p}d\mu_{g_0}ds 
+c\int_{t-r^2}^{t'}\int_M\left((\partial_s\psi^2)+|\nabla^{g_0}\psi|^2_{g_0}\right)u^{2p}d\mu_{g_0}ds.
\end{eqnarray*}

We need to control the integral involving the potential $\RP(g_0)$. By assumption, on the quadratic curvature decay at infinity:
\begin{eqnarray}\label{C^0-control-potential}
\|\RP(g_0)\|_{L^{\infty}(M\setminus B_{g_0}(x_0,r))}\leq\frac{C}{1+r^2},\quad r>0.
\end{eqnarray}
Let $\tau, \sigma\in (0,+\infty)$ such that $\tau+\sigma\leq r$. Define momentarily, 
\begin{eqnarray*}
P(x_0,t,r,s):=M\setminus B_{g_0}(x_0,r-s)\times(t-s^2,t], \quad 0<s^2\leq r^2<t.
\end{eqnarray*}
Notice that $s_1<s_2$ implies $P(x_0,t,r,s_1)\subset P(x_0,t,r,s_2).$

%In particular, $P(x, t,\sigma)\subset P(x,t,\tau+\sigma)\subset P(x,t,r)$.
Now, choose two smooth functions $\phi:\mathbb{R}_+\rightarrow[0,1]$ and $\eta:\mathbb{R}_+\rightarrow[0,1]$ such that
\begin{eqnarray*}
&& \supp(\phi)\subset [r-(\tau+\sigma),+\infty),\quad\phi\equiv 1\quad\mbox{in $[r-\tau,+\infty)$},\\
&&\quad \phi\equiv 0\quad\mbox{in $[0,r-(\tau+\sigma)]$},\quad 0\leq \phi'\leq c/\sigma,\\
&& \supp(\eta)\subset [t-(\tau+\sigma)^2,+\infty),\quad\eta\equiv 1\quad\mbox{in $[t-\tau^2,+\infty)$},\\
&& \eta\equiv 0\quad\mbox{in $(t-r^2,t-(\tau+\sigma)^2]$},\quad 0\leq \eta'\leq c/\sigma^2.
\end{eqnarray*}
Define $\psi(y,s):=\phi(d_{g_0}(x_0,y))\eta(s)$, for $(y,s)\in M\times(0,+\infty)$. Then,
\begin{eqnarray*}
&&\arrowvert\nabla^{g_0}\psi\arrowvert_{g_0}\leq\frac{c}{\sigma},\quad \arrowvert\partial_s\psi\arrowvert\leq \frac{c}{\sigma^2},
\end{eqnarray*}
for some uniform positive constant $c$.

In particular, thanks to claim (\ref{C^0-control-potential})
 applied to this cut-off function $\psi$ previously defined, one has :
\begin{eqnarray*}
&&\int_M(u^{2p}\psi^2)(t')d\mu_{g_0}+\int_{t-r^2}^{t'}\int_M|\nabla^{g_0}(\psi u^p)|_{g_0}^2d\mu_{g_0}ds\leq\\
&&c\left(\frac{p}{[1+r-(\tau+\sigma)]^2}+\frac{1}{\sigma^2}\right)\left(\int_{P(x_0,t,r,\tau+\sigma)}u^{2p}d\mu_{g_0}\right),\quad t'\in(t-r^2,t],
\end{eqnarray*}
which implies in particular that,
\begin{eqnarray}\label{sup-integral}
\sup_{t'\in(t-r^2,t]}\int_M(u^{2p}\psi^2)(t')d\mu_{g_0}\leq c\left(\frac{p}{[1+r-(\tau+\sigma)]^2}+\frac{1}{\sigma^2}\right)\left(\int_{P(x_0,t,r,\tau+\sigma)}u^{2p}d\mu_{g_0}\right).
\end{eqnarray}
Now, by Hölder inequality, for $s\geq 0$,
\begin{eqnarray*}
\int_M (\psi u^p)^{2+\frac{4}{n}}(s)d\mu_{g_0}\leq\left(\int_M (\psi u^p)^{\frac{2n}{n-2}}(s)d\mu_{g_0}\right)^{\frac{n-2}{n}}\left(\int_M (\psi u^p)^2(s)d\mu_{g_0}\right)^{\frac{2}{n}}.
\end{eqnarray*}
Therefore, to sum it up, if $\alpha_n:=1+2/n$, by using (\ref{sup-integral}) in the third line,
\begin{eqnarray*}
\int_{P(x_0,t,r,\tau)}\left(u^{2p}\right)^{\alpha_n}d\mu_{g_0}ds&\leq&\int_{P(x_0,t,r,0)}\left(\psi u^p\right)^{2\alpha_n}d\mu_{g_0}ds\\
&\leq&\int_{t-r^2}^t\left(\int_M (\psi u^p)^{\frac{2n}{n-2}}d\mu_{g_0}\right)^{\frac{n-2}{n}}\left(\int_M (\psi u^p)^2d\mu_{g_0}\right)^{\frac{2}{n}}ds\\
&\leq&c(n,g_0)\sup_{s\in(t-r^2,t]}\left(\int_M (\psi u^p)^2d\mu_{g_0}\right)^{\frac{2}{n}}\int_{M\times(t-r^2,t]}\arrowvert\nabla^{g_0}(\psi u^p)\arrowvert_{g_0}^2d\mu_{g_0}\\
&\leq&c(n,g_0)\left(\frac{p}{[1+r-(\tau+\sigma)]^2}+\frac{1}{\sigma^2}\right)^{\alpha_n}\left(\int_{P(x_0,t,r,\tau+\sigma)}u^{2p}d\mu_{g_0}\right)^{\alpha_n}.
\end{eqnarray*}
Define the following sequences : 
\begin{eqnarray*}
p_i:=\alpha_n^i,\quad \sigma_i:=2^{-1-i}(r/4),\quad \tau_{-1}:=3r/4,\quad\tau_i:=3r/4-\sum_{j=0}^i\sigma_j,\quad i\geq 0.
\end{eqnarray*}
Then, $\lim_{i\rightarrow +\infty}\tau_i= r/2$ and, for any $i\geq 0$,
\begin{eqnarray*}
\| u^2\|_{L^{p_{i+1}}(P(x_0,t,r,\tau_i))}\leq\left(c(n,g_0)\left(\frac{p_i}{1+[r-\tau_{i-1}]^2}+\frac{1}{\sigma_i^2}\right)\right)^{\frac{1}{p_i}}\| u^2\|_{L^{p_i}(P(x_0,t,r,\tau_{i-1}))},
\end{eqnarray*}
i.e.
\begin{eqnarray*}
\| u\|^2_{L^{\infty}(P(x_0,t,r/2))}\leq\Pi_{i=0}^{\infty}\left(c(n)\left(\frac{p_i}{1+[r-\tau_{i-1}]^2}+\frac{1}{\sigma_i^2}\right)\right)^{\frac{1}{p_i}}\| u\|^2_{L^{2}(P(x_0,t,r,3r/4))},
\end{eqnarray*}
since $P(x_0,t,r,r/2)=P(x_0,t,r/2)$.
 It remains to estimate the previous infinite product:
% It turns out that 
%\begin{eqnarray*}
%&&\sqrt{t-\tau_{i-1}^2}\geq\sqrt{t-r^2},\quad i\geq 0.\\
%\end{eqnarray*}

\begin{eqnarray*}
\Pi_{i=0}^{\infty}\left(\frac{p_i}{1+[r-\tau_{i-1}]^2}+\frac{1}{\sigma_i^2}\right)^{\frac{1}{p_i}}
&\leq&c(n)\frac{1}{r^{2+n}},
\end{eqnarray*}
i.e.
\begin{eqnarray}\label{L^2-bound-mean-value}
\sup_{P(x_0,t,r/2)}u^2&\leq&c(n,g_0) \frac{1}{r^{2+n}}
\int_{P(x_0,t,r,3r/4)}u^2d\mu_{g_0}ds.
\end{eqnarray}

To get a bound depending on the $L^1$ norm of $u$, one can proceed as in \cite{Li-Sch-Mea-Val} (the so called Li-Schoen's trick) by iterating (\ref{L^2-bound-mean-value}) appropriately. 

\end{proof}
\bibliographystyle{alpha.bst}
\bibliography{bib-ale-stability}

\end{document}